\documentclass[english,12pt,letter]{article}

\usepackage{amsthm}
\usepackage[latin9]{inputenc}
\usepackage{babel}
\usepackage[hmargin=0.9in,vmargin=1.25in]{geometry}
\usepackage{graphicx}
\usepackage{subfigure}
\usepackage[colorlinks=true,citecolor=blue,urlcolor=blue]{hyperref}
\usepackage{amsmath}
\usepackage{amssymb,amsfonts}

\newtheorem{thm}{Theorem}
\newtheorem{lem}{Lemma}

\newcommand{\Rnot}{\sigma_0}
\newcommand{\Sinf}{x_\infty}
\newcommand{\dom}{{\mathcal D}}

\newcommand{\ymax}{y_\text{max}}

\newenvironment{AMS}{\par\textbf{AMS subject classification.}}{\par}

\begin{document}
\title{Optimal control of an SIR epidemic through finite-time non-pharmaceutical intervention}
\author{
  David I. Ketcheson\thanks{Computer, Electrical, and Mathematical Sciences \& Engineering Division,
King Abdullah University of Science and Technology, 4700 KAUST, Thuwal
23955, Saudi Arabia. (david.ketcheson@kaust.edu.sa)}
}
\maketitle

\abstract{
We consider the problem of controlling an SIR-model epidemic
by temporarily reducing the rate of contact within a population.
The control takes the form of a multiplicative reduction in the contact rate
of infectious individuals.  The control is allowed to be applied only over
a finite time interval, while the objective is to minimize the total number of
individuals infected in the long-time limit, subject to some cost function for
the control.  We first consider the no-cost scenario and analytically determine
the optimal control and solution.  We then study solutions when a cost of intervention
is included, as well as a cost associated with overwhelming the available medical resources.
Examples are studied through the numerical solution of the associated Hamilton-Jacobi-Bellman
equation.  Finally, we provide some examples related directly to the current pandemic.
}

\begin{AMS}
    92D30, 34H05, 49N90, 92-10, 49L12
\end{AMS}

\section{Problem description and assumptions}
The classical SIR model of Kermack \& Mckendrick \cite{kermack1927contribution} is
\begin{subequations} \label{SIR}
\begin{align} 
    x'(t) & = -\gamma \Rnot y(t) x(t) \label{eq:x} \\
    y'(t) & = \gamma \Rnot y(t) x(t) - \gamma y(t) \label{eq:y} \\
    (x(0),y(0)) & \in \dom := \{(x_0,y_0) : x_0 > 0, y_0 > 0, x_0+y_0 \le 1\},
\end{align}
\end{subequations}
where $x(t), y(t)$ represent the susceptible and infected populations
respectively, while the recovered population is $z(t)=1-x(t)-y(t)$.  The region
$\dom$ is forward-invariant and a unique solution exists for all time
\cite{hethcote2000mathematics}.  While
the temporal dynamics of \eqref{SIR} depend on both $\Rnot$ and $\gamma$, the set
of trajectories depends only on the basic reproduction number 
$\Rnot$.  Dynamics for two values of $\Rnot$ are
shown in Figure \ref{fig:dynamics}.  

The system \eqref{SIR} is at equilibrium if $y(t)=0$.  This equilibrium is stable if and
only if $x(t)\le 1/\Rnot$, a condition referred to as {\em herd immunity}.  If
this condition is not satisfied at
the initial time, then $y(t)$ will first increase until it is, and then decrease,
approaching zero asymptotically.  The SIR model assumes that recovery
confers permanent immunity.

\begin{figure}
    \centering
    \subfigure[$\sigma=3$]{\label{sigma3}\includegraphics[width=0.49\textwidth]{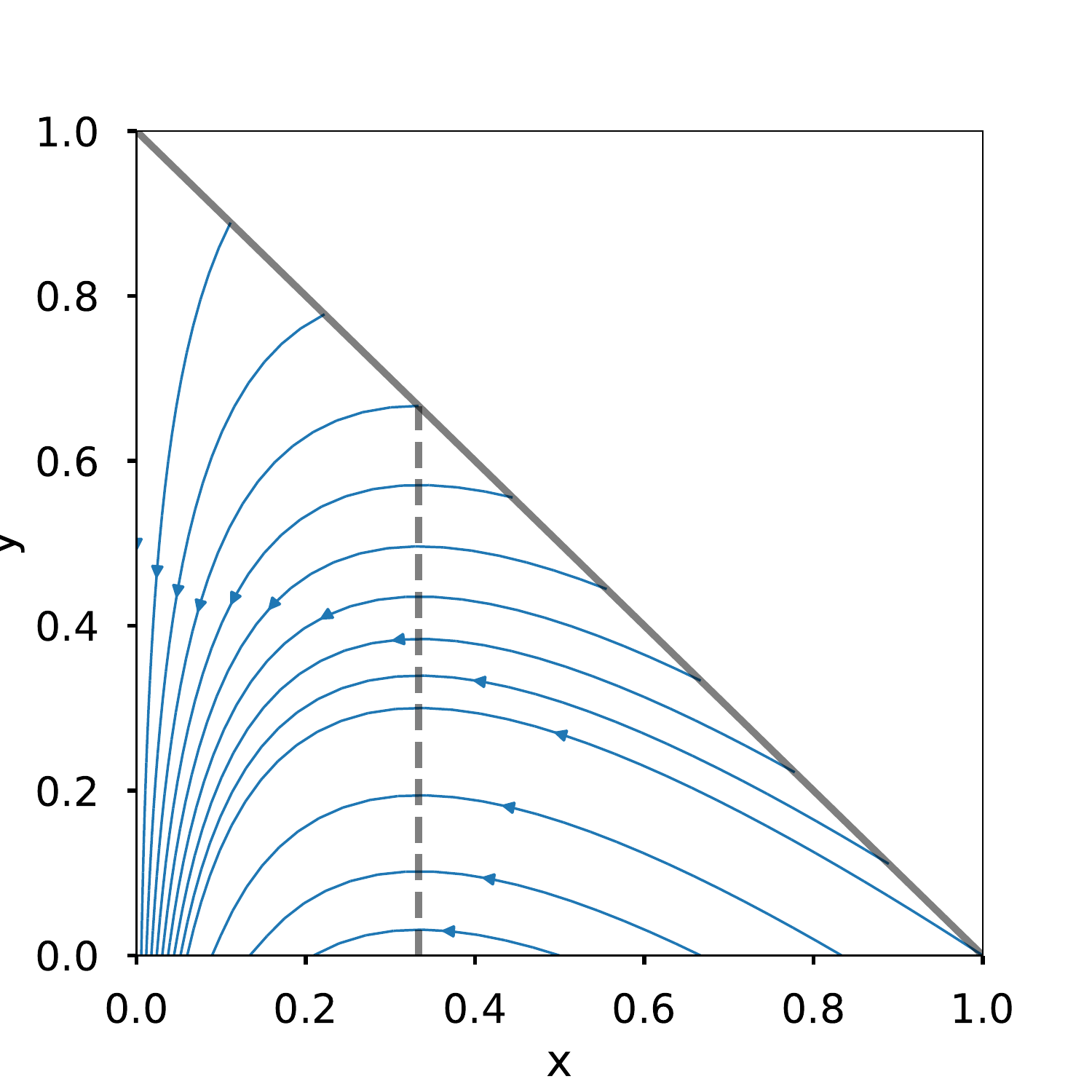}}
    \subfigure[$\sigma=1.5$]{\label{sigma15}\includegraphics[width=0.49\textwidth]{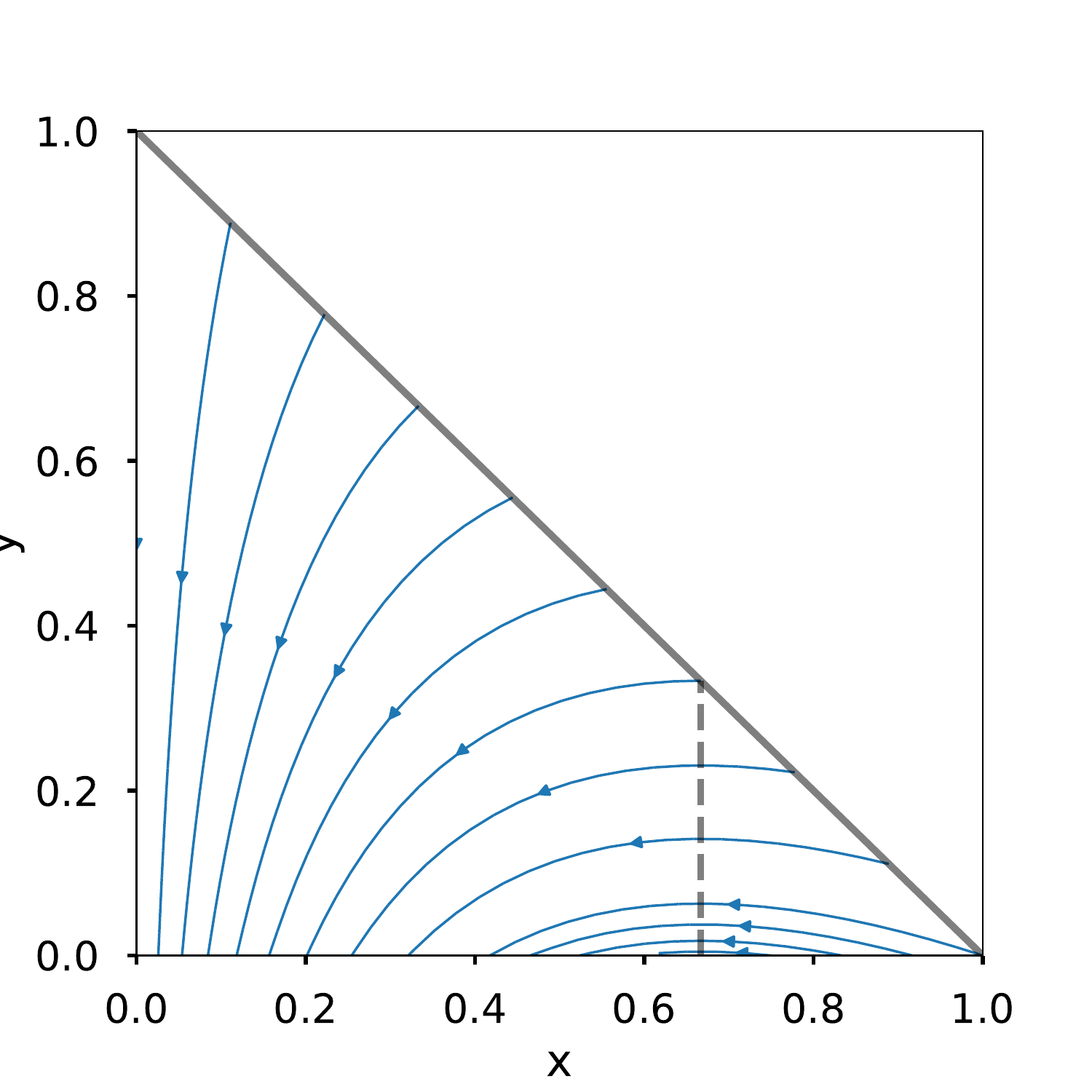}}
    \caption{Dynamics of the SIR model \eqref{SIR} for two values of the basic reproduction number.
            The critical value $x=1/\sigma$ is shown with a dashed line.\label{fig:dynamics}}
\end{figure}

For many diseases affecting humans, herd immunity is achieved
through vaccination of a sufficient portion of the population.  Herein
we assume a vaccine is unavailable, so that herd immunity can only be achieved
through infection and recovery.
Our goal is to minimize $z_\infty := \lim_{t \to \infty} z(t)$, or equivalently
(since $y_\infty=0$)
to maximize the long-time limit of the susceptible fraction:
$\Sinf = \lim_{t\to\infty} x(t)$.
This has the effect of minimizing
the number of eventual deaths, which would be proportional to $z_\infty$.

This is equivalent to minimizing the number of deaths, if we assume that
some fixed fraction of the recovered population $z(t)$ dies from the disease.
From the foregoing it is clear that $\Sinf \le 1/\Rnot$.  The difference
$1/\Rnot-\Sinf$ is referred to as {\em epidemiological overshoot}.
For COVID-19, a review of early estimates of $\Rnot$ can be found in
\cite[Table 1]{liu2020reproductive}, with mean 3.28 and median 2.79.
In accordance with these estimates, we use a value $\Rnot=3$ in most of
the examples in this work.  With this value, the SIR model implies that eventually
at least two-thirds of the world population will eventually have COVID-19 antibodies;
this number is likely to be significantly higher in reality due to epidemiological
overshoot.  For instance, it can be seen from Figure \ref{sigma3} that, starting
from a fully susceptible population and a small number of infected individuals,
in the absence of control the SIR model predicts that over $90\%$ of the population
would be infected.

This overshoot can be reduced through non-pharmaceutical intervention (NPI),
which is simply a means to reduce contact between infected and susceptible
individuals; reductions of this kind occurred for instance as a result
of NPIs imposed during the 1918 flu pandemic \cite{bootsma2007effect}.  We model a NPI
control via a time-dependent reproduction number $\sigma(t)\in[0,\Rnot]$ with the system
\begin{subequations} \label{SIRq}
\begin{align}
    x'(t) & = -\gamma \sigma(t) y x \\
    y'(t) & = \gamma \sigma(t) y x - \gamma y \\
    (x(0),y(0)) & \in \dom := \{(x_0,y_0) : x_0 > 0, y_0 > 0, x_0+y_0 \le 1\}.
\end{align}
\end{subequations}
A temporary reduction in $\sigma$ can account for both
population-wide interventions and interventions specific to identified infectious
(or possibly infectious) individuals.  The SIR model with a time-dependent
reproduction number (or equivalently, a time-dependent contact rate) has been 
considered before, for instance in \cite{bootsma2007effect,sun2020tracking}.

Typically, an epidemic does not result in substantial {\em permanent} change in the contact rate of
a population.  We therefore assume 
\begin{align} \label{q-shortterm}
    \sigma(t)=\Rnot \text{ for } t>T,
\end{align}
i.e., that intervention can only be applied over a finite interval $t \in [0,T]$.
Since $x_\infty=1/\sigma_0$ only at the single point $(x=1/\sigma_0,y=0)$, and since 
the $y=0$ axis cannot be reached in a finite time, \eqref{q-shortterm} implies
that any solution must have $x_\infty<1/\sigma_0$.

We state the control problem as follows:
\begin{align} \label{eq:inf-time-problem}
\begin{aligned}
& \text{Given } (x_0, y_0)\in\dom, \sigma_0>0, T>0,  \\
& \text{ choose an admissible control } \sigma(t): [0,T] \to [0,\Rnot] \text{ to minimize }  \\
&     J(x(t),y(t),\sigma(t)) = -\lim_{t\to\infty} x(t) + \int_0^T L(x(t),y(t),\sigma(t)) dt \\
& \text{ subject to } \eqref{SIRq}.
\end{aligned}
\end{align}
Here $J$ is the objective function that accounts for the desire to
minimize infections as well as a running cost of imposing control.
We assume throughout that $L$ is convex with respect to $q(t)=1-\sigma/\Rnot$ and
bounded uniformly by a constant for all $(x,y)\in \dom$, $\sigma\in[0,\sigma_0]$.

There is a large body of work on compartmental epidemiological models
and control for such models; see e.g. \cite{hethcote2000mathematics,lenhart2007optimal}
and references therein.  A number of works focus on optimal control
through vaccination; see e.g. \cite{kar2011stability}.
Other works, such as \cite{yan2008optimal,safi2013dynamics,agusto2013optimal}
focus on explicit modeling of and/or control through quarantined and isolated individuals.
A review of work on optimal control in compartmental epidemiological models is
presented in \cite{sharomi2017optimal},
along with the formulation of necessary conditions (based on
Pontryagin's maximum principle) for various extensions of the SIR model.
For modeling and control based on even more detailed models incorporating
spatial spread and human networks, see e.g. \cite{ferguson2005strategies}.

\subsection{Objectives and contributions}

The modeling and assumptions in the present work are motivated by the
current COVID-19 epidemic, which so far is being managed through broad
NPIs and without a vaccine.  In order to understand the effects of NPIs
imposed on an entire population, we stick to the simple model \eqref{SIRq}
rather than explicitly modeling quarantined individuals.
Since such population-wide measures cannot be maintained indefinitely, we
invoke the finite-time control assumption \eqref{q-shortterm}.
This assumption is not new (see e.g. \cite{greenhalgh1988some}),
but unlike previous works our objective function is still based on the
long-term outcome (rather than the outcome at time $T$).
This drastically changes the nature of optimal solutions.

Although the broad motivation for this work comes from the current epidemic,
our primary objective is to understand general properties of optimal controls
for the variable-$\sigma$ SIR system \eqref{SIRq}.  To this end, we also
investigate solutions in certain
asymptotic regimes (such as when there is little or no cost associated
with the control).  Nevertheless, the values of the key parameters $\gamma$
and $\Rnot$ for all examples are chosen to fall in the range of current estimates for
COVID-19.

One novel aspect of this work is that the problem is posed in terms of the
infinite-time limit, but formulated in a way that only requires solution
over a finite time interval.  Indeed, without this reformulation we found
that the problem was extremely ill-conditioned; this reformulation is also
needed in order to compute approximate solutions via a Hamilton-Jacobi-Belmman
equation.
This reformulation is presented in Section \ref{sec:prelims}.  The main theoretical
result is an exact characterization of the optimal control when $L=0$, given
as Theorem \ref{thm:no-cost} in Section \ref{sec:analytic}.  

Typical results in the literature on control of compartmental epidemiological models
are numerical and are based on Pontryagin's weak maximum principle, which gives only
necessary conditions for optimality.  At best, uniqueness is shown for small times; see e.g. 
\cite{kirschner1997optimal,fister1998optimizing,jung2002optimal,yan2008optimal,kar2011stability,sharomi2017optimal}.
In contrast, here the main result includes a proof of optimality for arbitrarily large times.
In Section \ref{sec:exploration} we explore
the behavior of optimal solutions for $L\ne 0$ under various interesting
cost functions and parameter regimes.  Here the results are based on
solutions of the relevant Hamilton-Jacobi-Bellman equation, which is
both necessary and sufficient for optimality.
In Section \ref{sec:application} we consider direct application to the COVID-19
pandemic.  Some conclusions are drawn in Section \ref{sec:conclusion}.

The code for all examples in this paper is publicly available \cite{ketcheson2021SIRRepro}.

\section{Formulation over a finite time interval\label{sec:prelims}}
In this section we reformulate the control problem \eqref{eq:inf-time-problem} in terms of
the solution over a finite time interval $[0,T]$.  This reformulation is
necessary both to facilitate the exact solution in Section \ref{sec:analytic}
and to arrive at a numerically-tractable problem for computing
approximate solutions, as described in Section \ref{sec:exploration}.


In general, the solution of \eqref{SIRq} depends on the initial data
$(x_0,y_0)$, the control $\sigma(t)$, and time $t$, so it is natural to write
$x(t;\sigma(t),x_0,y_0)$.
In what follows it will be convenient to make a slight abuse of notation and
write $x(t;\sigma(t))$ or $x(t)$ when there is no chance of confusion.

For a fixed reproduction number, the asymptotic susceptible fraction $\Sinf$
can be obtained from the solution $x(t), y(t)$ at any time $t$, since solutions
of \eqref{SIR} move along contours of $\Sinf$.  Thus we will write $\Sinf(x,y)$
or $\Sinf(x,y,\Rnot)$.

\subsection{A formula for $\Sinf$}
In this subsection we review the solution of the SIR model without
control \eqref{SIR}.
It can be shown that $x(t)$ satisfies (see \cite{harko2014exact,pakes2015lambert} and
\cite[pp.707-708]{kermack1927contribution})
$$
    x(t)e^{\Rnot z(t)} = x_0 e^{\Rnot z_0}.
$$
Since $z=1-x-y$ we define
$$
   \mu(x,y,\Rnot) := x(t) e^{-\Rnot(x(t)+y(t))},
$$
which is constant in time for any solution of \eqref{SIR}.
The trajectories in Figure \ref{fig:dynamics} are thus also contours of $\mu$.
Since $y_\infty=0$, we have
$$
    x_\infty = x_0 e^{\Rnot(x_\infty-x_0-y_0)} = \mu(x_0,y_0,\Rnot) e^{\Rnot x_\infty}.
$$
Setting $w=-x_\infty \Rnot$ we have
$$
    we^w = -x_0 \Rnot e^{-\Rnot(x_0+y_0)} = -\mu \Rnot.
$$
Thus $w = W_0(-\mu\Rnot)$ where $W_0$ is the principal branch of Lambert's $W$-function \cite{pakes2015lambert},
and
\begin{align} \label{eq:xinf}
    x_\infty(x,y,\sigma_0) = -\frac{1}{\Rnot}W_0(-\mu(x,y,\sigma_0) \Rnot).
\end{align}
Formula \eqref{eq:xinf} allows us to rewrite the problem \eqref{eq:inf-time-problem} in
terms of the state at time $T<\infty$:
\begin{align} \label{eq:basic-problem}
\begin{aligned}
& \text{Given } (x_0, y_0) \in \dom, \sigma_0>0, T>0,  \\
& \text{ choose an admissible control } \sigma(t): [0,T] \to [0,\Rnot] \text{ to minimize }  \\
&     J = -x_\infty(x(T),y(T),\sigma_0) + \int_0^T L(x(t),y(t),\sigma(t)) dt \\
& \text{ subject to } \eqref{SIRq}.
\end{aligned}
\end{align}

In what follows we will also require the derivatives of $\Sinf$ with respect to $x$, $y$, and $\mu$.
Direct computation gives
\begin{subequations} \label{xinf-grad}
\begin{align}
    \frac{\partial \Sinf}{\partial y(t)} & = -\frac{\Rnot \Sinf}{1-\Rnot \Sinf} \\
    \frac{\partial \Sinf}{\partial x(t)} & = \left(1-\frac{1}{x(t)\Rnot}\right) \frac{\partial \Sinf}{\partial y(t)}
      = \frac{1-\Rnot x(t)}{1-\Rnot \Sinf} \cdot \frac{\Sinf}{x(t)} \\
    \frac{\partial \Sinf}{\partial \mu} & = \frac{e^{\Rnot\Sinf}}{1-\Rnot \Sinf}.
\end{align}
\end{subequations}
Using these expressions we can also compute the rate of change of $\Sinf$ when some
control $\sigma(t)$ is applied:
\begin{align} \label{dxinf-dt}
    \frac{\partial \Sinf}{\partial t} = \frac{\gamma y \Sinf}{1-\Rnot\Sinf}(\Rnot - \sigma(t)).
\end{align}
From this we see that the impact of an intervention on $\Sinf$ is independent of $x(t)$ and
directly proportional to $y(t)$.  This indicates that intervention is more impactful
when there is a larger infected population.
%
%


\subsection{Bounds on $\Sinf$}
Now we turn our attention to the SIR system with control \eqref{SIRq}.
Henceforth we assume that $\sigma(t)\in[0,\sigma_0]$ for almost every $t\in[0,T]$;
we say that such a control is \emph{admissible}.


It is straightforward to show that \eqref{SIRq} has a unique solution for all time
for any initial data in $\dom$ and any admissible control, by the same
arguments used for \eqref{SIR}.
The proof of the next lemma shows that applying any control $\sigma(t)<\Rnot$ over
any length of time leads to an increase in $x_\infty$.
\begin{lem}
Let $\Rnot>0$ and $(x_0,y_0)\in\dom$ be given. 
Let $\sigma(t)$ be an admissible control.  Then for $t\ge0$ we have
$$
    x_\infty(x(t;\sigma(t)),y(t;\sigma(t)),\Rnot) \ge x_\infty(x_0,y_0,\Rnot).
$$
\end{lem}
\begin{proof}
    Dividing \eqref{eq:y} by \eqref{eq:x} gives
    \begin{align} \label{eq:dydx}
        \frac{dy}{dx} = -1 + \frac{1}{\sigma(t) x}.
    \end{align}
    Thus reducing $\sigma(t)$ has the effect of increasing $dy/dx$.
    Since all trajectories flow to the left ($x$ is a decreasing function of $t$),
    this means that the solution trajectory obtained with $\sigma(t)$ lies
    below that obtained with $\sigma_0$, for all $t>0$.  Since
    $\Sinf$ is a decreasing function of $y$, this completes the proof.
\end{proof}

Thus for any admissible control and any initial data we have
$$
\Sinf(x_0,y_0,\Rnot) \le \Sinf(x(T),y(T),\Rnot) \le 1/\Rnot.
$$

\subsection{Existence and necessary conditions for an optimal control\label{sec:pmp}}
Let us define the Hamiltonian
\begin{align} \label{eq:ham}
    H(x(t),y(t), \sigma(t), \lambda_{1,2}(t), t) = -\lambda_1(t) \gamma \sigma(t) y(t) x(t) + \lambda_2(t)\gamma y(t)(\sigma(t) x(t) - 1) + L(x(t),y(t),\sigma(t)),
\end{align}
and the adjoint variables $\lambda_1(t), \lambda_2(t)$, which are required to satisfy
\begin{subequations}\label{lambda-odes}
\begin{align} 
    \lambda_1'(t) & = -\frac{\partial H}{\partial x} = (\lambda_1-\lambda_2)\gamma\sigma(t) y(t) - \frac{\partial L}{\partial x} \\
    \lambda_2'(t) & = -\frac{\partial H}{\partial y} = (\lambda_1-\lambda_2)\gamma\sigma(t) x(t) + \lambda_2 \gamma - \frac{\partial L}{\partial y} \\
    \lambda_1(T) & = -\frac{\partial \Sinf(T)}{\partial x} =\frac{\partial }{\partial x(T)} (-x_\infty(x(T),y(T),\Rnot) = \left(1-\frac{1}{x(T)\Rnot}\right)\lambda_2(T) \label{bc1} \\
    \lambda_2(T) & = -\frac{\partial \Sinf(T)}{\partial y} = \frac{\partial }{\partial y(T)} (-x_\infty(x(T),y(T),\Rnot), \label{bc2}
\end{align}
\end{subequations}
where $x(t), y(t)$ satisfy $\eqref{SIRq}$.
Note that the final conditions for $\lambda_{1,2}$ can be computed from \eqref{xinf-grad}.
We have the following result.

\begin{thm}
    Let $(x_0, y_0)\in\dom$ and $\Rnot, \gamma, T\ge 0$ be given.  
    Let the running cost $L$ be given such that it is convex with respect to $q$,
    bounded uniformly by a constant for all $(x,y)\in \dom$, $\sigma\in[0,\sigma_0]$,
    and continuously differentiable with respect to $x$ and $y$.
    Then there
    exists an admissible control $\sigma^*(t)$ for \eqref{eq:basic-problem} and corresponding
    response $(x^*(t),y^*(t))$ such that $J$ is minimized over the set of
    admissible controls.  Furthermore, there exist adjoint functions
    $\lambda_{1,2}(t)$ satisfying \eqref{lambda-odes} for almost all
    $t\in[0,T]$ with $x(t)=x^*(t), y(t)=y^*(t)$, and
    such that the Hamiltonian is minimized pointwise with respect to $\sigma$:
    \begin{align} \label{pointwise-min}
        H(x^*(t),y^*(t),\sigma^*(t),\lambda_{1,2}(t),t) = \inf_{\sigma\in[0,\sigma_0]} H(x,y,\sigma,\lambda_{1,2},t)
    \end{align}
    for almost all $t \in[0,T]$.
\end{thm}
\begin{proof}
    The existence of an optimal control is guaranteed by \cite[Theorem 23.11]{clarke2013functional}
    since $L$ is convex with respect to $q(t)=1-\sigma(t)/\Rnot$, the state solutions $(x(t),y(t))$
    and their derivatives in time
    are bounded, the system \eqref{SIRq} is Lipschitz with respect to $x, y$, and the control $\sigma(t)=\sigma_0$
    is admissible and leads to a finite cost.
    The second part of the Theorem follows from applying Pontryagin's weak
    maximum principle as stated e.g. in \cite[Theorem 22.2]{clarke2013functional},
    which applies due to the assumptions on $L$ and since $\Sinf$ is continuously differentiable with respect to $x$ and $y$.
\end{proof}
    Observe that condition \eqref{pointwise-min} implies that
    the optimal control $\sigma^*(t)$ satisfies the optimality condition
    \begin{align} \label{eq:sigma-c2}
        \sigma^*(t) = \max\left(0,\min\left(\sigma_0,\hat{\sigma}(t)\right)\right),
    \end{align}
    where
    \begin{align} \label{eq:dldsig}
        \left. \frac{\partial L}{\partial \sigma}\right|_{\sigma(t)=\hat{\sigma}(t)} = - (\lambda_2(t)-\lambda_1(t))\gamma y x.
    \end{align}

\subsection{Infinite-time control}
In this section only, we consider controls that reach the optimal value $\Sinf = 1/\Rnot$.
This is achieved only at $(x,y)=(1/\Rnot,0)$, a state that cannot be reached
from any other state without imposing some control, and which in any case can
only be reached after an infinite time.  Thus we momentarily set aside the restriction
\eqref{q-shortterm} and consider controls extending up to an arbitrarily large time $T$.
We still require that the system approach a stable equilibrium point as $t\to\infty$.
We assume that $x_0\ge1/\Rnot$, since otherwise the maximum achievable value of $\Sinf$
is $x_0$, which would be achieved by taking simply $\sigma(t)=0$ for all $t$.
We also take $L=0$ so that an optimal control is any control satisfying
$$
    \lim_{t \to \infty} x(t,\sigma(t)) = 1/\Rnot.
$$
There are infinitely many such controls.  Two are particularly simple and
are of interest.

The first is a constant control $\sigma(t) = \sigma_*(x_0, y_0, \Rnot)$.
By \eqref{eq:xinf} we must have $\Sinf(x_0,y_0,\sigma_*)=1/\Rnot$, so $\sigma_*$ is the solution of
$$
    W_0(-\mu(x_0,y_0,\sigma_*)\sigma_*) = -\frac{\sigma_*}{\sigma_0}.
$$

The second is a bang-bang control in which
$$
    \sigma(t) = \begin{cases} \Rnot & x>1/\Rnot \\ 0 & x=\Rnot. \end{cases}
$$
The response for each of these controls is shown for a specific example in
Figure \ref{fig:two-controls}.
\begin{figure}
    \centering
    \includegraphics[width=0.5\textwidth]{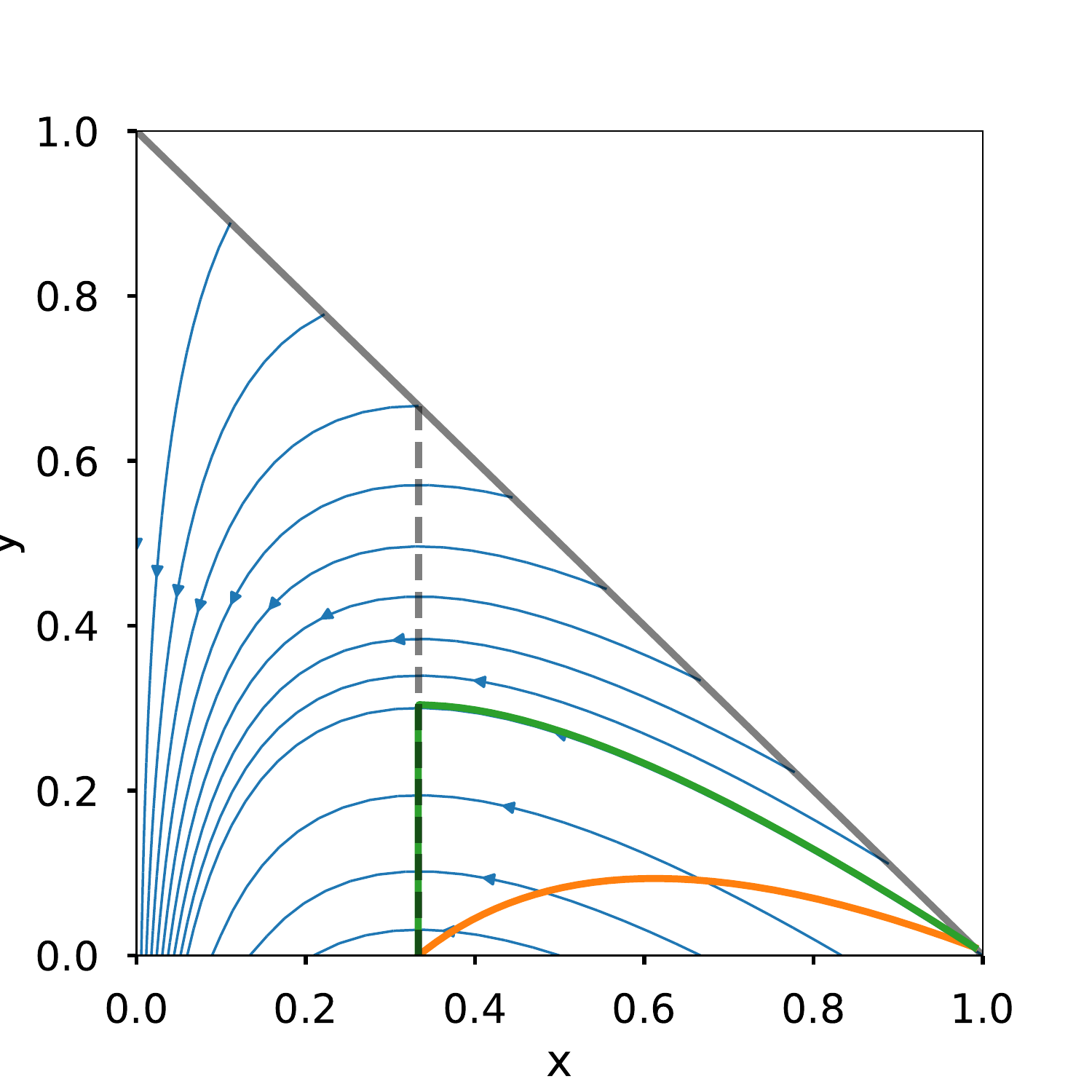}
    \caption{Two infinite-time controls that give $\Sinf=1/\Rnot$.  Here $\Rnot=3$ and
        $(x_0,y_0)=(0.99,0.01)$.  For the constant control, $\sigma(t)=\sigma_*\approx(1-0.4557)\Rnot$.\label{fig:two-controls}}
\end{figure}

\section{Optimal control with $L=0$\label{sec:analytic}}
In this section we derive the exact solution of the control problem
\eqref{eq:basic-problem} with $L=0$ (i.e., when the goal of increasing $x_\infty$ completely
trumps any associated costs or other concerns).  Then \eqref{eq:basic-problem} becomes
\begin{align} \label{eq:no-cost-problem}
\begin{aligned}
& \text{Given } (x_0, y_0)\in \dom, \sigma_0>0, T>0, \\
& \text{ choose an admissible control } \sigma(t): [0,T] \to [0,\Rnot] \\
& \text{ to minimize }  J = -x_\infty(x(T),y(T),\sigma_0) \\
& \text{ subject to } \eqref{SIRq}.
\end{aligned}
\end{align}
This problem can be reformulated as a minimum-time control problem.  

\begin{lem} \label{lem:min-time}
Let $\sigma^*(t)$ be an optimal control for \eqref{eq:no-cost-problem}, and
let $(x^*(T),y^*(T))$ denote the corresponding terminal state.
Then there is no admissible control that reaches $(x^*(T),y^*(T))$ from $(x_0,y_0)$
before time $T$.
\end{lem}
\begin{proof}
Suppose there were a control $\hat{\sigma}(t)$ that leads to $(x(t^*),y(t^*)) = (x^*(T),y^*(T))$ for some $t^*<T$.  Then
we could obtain a smaller value of $J$ in \eqref{eq:no-cost-problem} by using
$\hat{\sigma}$ up to time $t^*$ combined with the choice $\sigma(t)=0$ for $t>t^*$.
This contradicts the optimality of $\sigma^*(t)$.
\end{proof}

Furthermore, the optimal control must be a bang-bang control.
\begin{lem} \label{lem:bang-bang}
Let $\sigma(t)$ be an optimal control for \eqref{eq:no-cost-problem}.
Then
\begin{align} \label{eq:switch}
    \sigma(t) & = \begin{cases} 0 & \lambda_1(t)<\lambda_2(t) \\ \sigma_0 & \lambda_1(t) > \lambda_2(t) \end{cases}
\end{align}
where $\lambda_{1,2}(t)$ are given by \eqref{lambda-odes}.
\end{lem}
\begin{proof}
From \eqref{eq:ham} with $L=0$, we have
$$
\frac{\partial H}{\partial \sigma} = (\lambda_2(t)-\lambda_1(t))\gamma y(t) x(t).
$$
The optimality condition then implies \eqref{eq:switch} except at points where $\partial H/\partial \sigma=0$
(see e.g. \cite[Ch. 17]{lenhart2007optimal}.

It remains to show that there are no singular arcs.
Since $x(t),y(t) >0$ for $t<\infty$, we have that $\partial H/\partial \sigma=0$
if and only if $\lambda_1=\lambda_2$.  Suppose (by way of contradiction) that
the latter condition holds on an open interval.  Then on that interval we would
have (by \eqref{lambda-odes} with $L=0$):
\begin{align}
\lambda_1'(t) & = (\lambda_1-\lambda_2)\gamma\sigma y = 0 \\
\lambda_2'(t) & = (\lambda_1-\lambda_2)\gamma\sigma x = 0.
\end{align}
By continuity, this would imply that $\lambda_1(t)=\lambda_2(t)$ over
the whole interval $[0,T]$, and in particular at time $T$.
But then \eqref{bc1}-\eqref{bc2} gives
$$
\left(1-\frac{1}{x(T)\Rnot}\right)\lambda_2(T) = \lambda_2(T).
$$
We know from \eqref{xinf-grad} that $\lambda_2(T)\ne 0$, so this
is a contradiction.
\end{proof}

This motivates the following lemma.

\begin{lem} \label{lem:shortest-path}
Let $(x_0,y_0)$ and $(x_1,y_1)$ be given such that $x_0, x_1 \ge 1/\Rnot$ and
$\Sinf(x_0,y_0,\Rnot)\ge\Sinf(x_1,y_1,\Rnot)$.
Let $\sigma(t)$ be a bang-bang control such that $(x(t_1;x_0,y_0,\sigma(t)),y(t_1;x_0,y_0,\sigma(t)))=(x_1,y_1)$
for some $t_1\ge0$.  Then the minimum value of $t_1$ is achieved by taking
\begin{align} \label{eq:one-switch}
    \sigma(t) & = \begin{cases}
        \Rnot & t<t^* \\
        0 & t^* \le  t \le t_1,
    \end{cases}
\end{align}
where $t^*$ satisfies $x(t^*;x_0,y_0,\Rnot)=x_1$.
\end{lem}
\begin{proof}
Since $\sigma(t)$ is a bang-bang control, the trajectory $(x(t;\sigma(t)),y(t;\sigma(t)))$ 
consists of a sequence of segments each of which is a solution of
\eqref{SIRq} with $\sigma=0$ (traveling directly downward) or with $\sigma=\sigma_0$
(traveling along a contour of $x_\infty$).  Some trajectories of this type are
illustrated in Figure \ref{fig:bangbangtraj}.  Notice that each trajectory must
traverse the same distance in the x-direction; since $x'(t)=-\beta xy$ this
travel is faster at larger $y$ values.  Meanwhile, the total length of all the
downward ($\sigma=0$) segments is the same for any trajectory, and since for
these segments $y'(t) = -\gamma y$, travel is again faster at larger $y$ values.
The control given in the lemma makes all these traversals at the largest
possible values of $y$, so it arrives in the shortest time.
\end{proof}

Combining these three lemmas, we obtain the following.
\begin{thm} \label{thm:one-switch}
    Any optimal control for \eqref{eq:no-cost-problem} is of the form \eqref{eq:one-switch}
    with $t_1=T$.
\end{thm}
\begin{proof}
    By Lemmas \ref{lem:min-time} and \ref{lem:bang-bang}, the optimal control must be bang-bang and must solve
    the optimal-time problem.  Then Lemma \ref{lem:shortest-path} applies and gives the stated result.
\end{proof}

\begin{figure}
    \centering
    \includegraphics[width=0.5\textwidth]{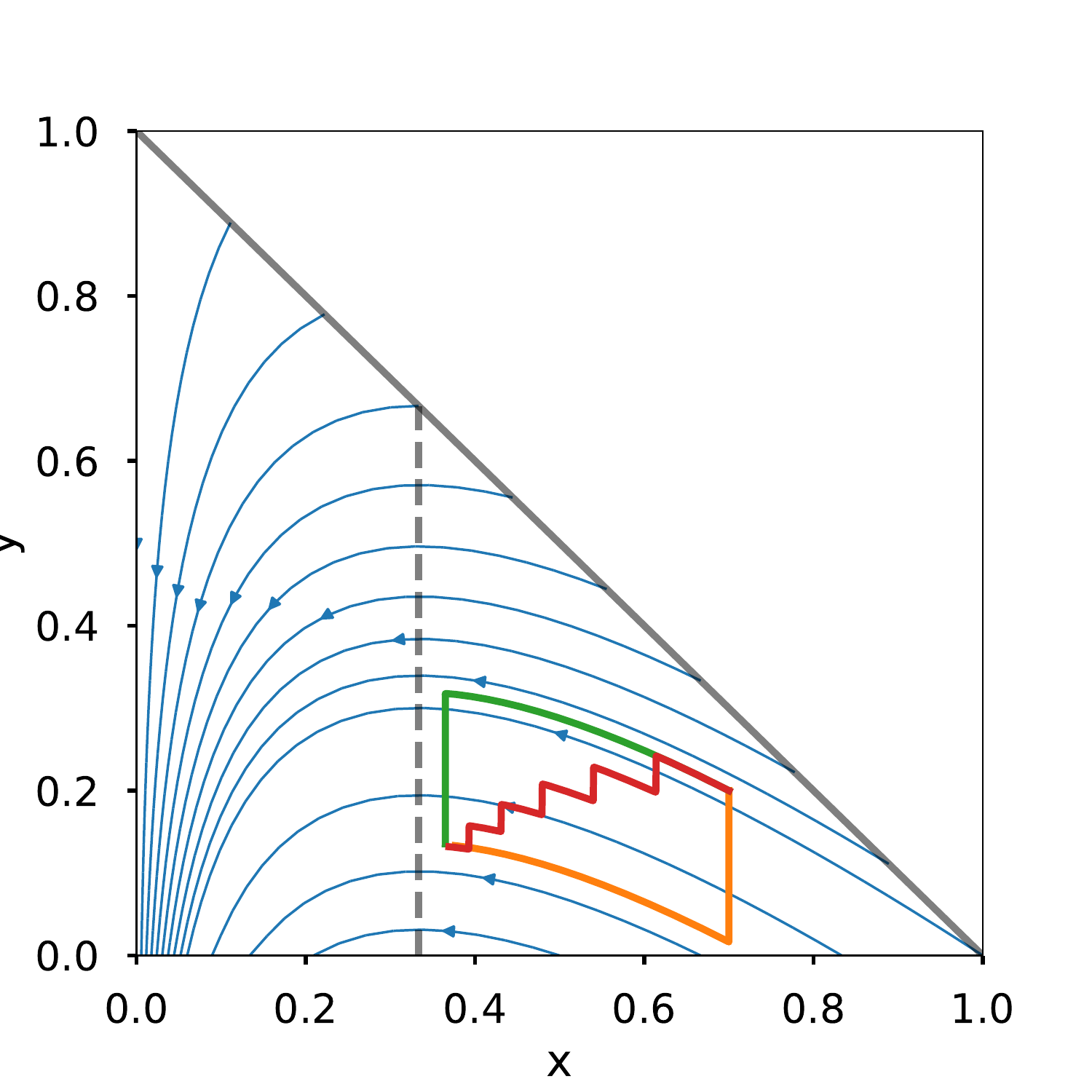}
    \caption{Three different paths between two states, each obtained
    with a bang-bang control.  The top (green) path arrives in the shortest time.\label{fig:bangbangtraj}}
\end{figure}
We can now give the solution of \eqref{eq:no-cost-problem}.
\begin{thm} \label{thm:no-cost}
The optimal control for \eqref{eq:no-cost-problem} is unique and is given by
\begin{align}
    \sigma(t) & = \begin{cases}  
        \Rnot & t<t^* \\
        0 & t^* \le  t \le T,
    \end{cases}
\end{align}
where 
\begin{align} \label{eq:dropcond}
    t^*=0 \text{ if } x_0\le\frac{1}{\sigma_0(1-e^{-\gamma T})},
\end{align}
and otherwise $t^*$ is the unique solution of
\begin{align} \label{xtstar}
    x(t^*;\Rnot,x_0,y_0) = \frac{1}{\sigma_0(1-e^{-\gamma(T-t^*)})}.
\end{align}
\end{thm}
\begin{proof}
First, suppose $x(0)\le1/\sigma_0$.  The claimed optimal control gives $x(T)=x_0$, whereas
any other control will give $x(T)<x_0$.  Similarly, we see from \eqref{SIRq} that
the optimal control gives $y(T)=e^{-\gamma T}y_0$ and any other control will
lead to a larger value of $y(T)$.  Since $\Sinf$ is a decreasing function of $y$ and
(for $x<1/\Rnot$) an increasing function of $x$, the proposed control is optimal in this case.

Now suppose $x(0)>1/\sigma_0$.  We reformulate the objective as follows.
From \eqref{xinf-grad} we see that $\Sinf$ is a strictly monotone increasing function of $\mu$,
so that maximizing $\Sinf$ is equivalent to maximizing $\mu$. Now
\begin{align*}
    \mu'(t) & = (x'(t)-\Rnot x(t)(x'(t)+y'(t)))e^{-\Rnot(x(t)+y(t))} \\
            & = (\Rnot - \sigma(t))\gamma x(t) y(t) e^{-\Rnot(x(t)+y(t))} \\
            & = \gamma y(t) (\Rnot-\sigma(t))\mu(t).
\end{align*}
Thus
$$
    \mu(t) = \exp\left(\gamma \int_0^t y(\tau) (\Rnot-\sigma(\tau)) d\tau\right) \mu(0).
$$
Thus, maximizing $\Sinf(T)$ is equivalent to maximizing
$$
    I := \int_0^T y(\tau) (\Rnot-\sigma(\tau)) d\tau.
$$
From Theorem \ref{thm:one-switch} we have that
\begin{align*}
    I & = \int_{t^*}^T y(\tau) \Rnot d\tau \\
      & = \frac{\Rnot}{\gamma} y(t^*) \left(1-e^{-\gamma(T-t^*)}\right).
\end{align*}
Differentiating with respect to $t^*$ gives
\begin{align} \label{eq:didt}
    \frac{d I}{dt^*} & = \Rnot y(t^*)\left( \Rnot x(t^*)(1-e^{-\gamma(T-t^*)}) - 1 \right).
\end{align}
If the inequality in \eqref{eq:dropcond} is satisfied then this has no
zero and $I$ is maximized by taking $t^*=0$.  If the condition in
\eqref{eq:dropcond} is not satisfied, then setting the right hand side
of \eqref{eq:didt} equal to zero yields the condition \eqref{xtstar}.
By checking the second derivative, it is easily confirmed that this
is indeed a maximum.
\end{proof}

We remark that the above result apparently cannot be obtained via standard
sufficiency conditions based on Pontryagin's maximum principle, due to the
nonconvexity of the right hand side of the SIR system \eqref{SIRq}.

Some optimal solutions for particular instances of \eqref{eq:no-cost-problem}
are shown in Figures \ref{fig:example1} and \ref{fig:diff-time-opt},
all with the same initial data and parameters $\beta, \gamma$ but with different final
times $T$.  Allowing for a longer intervention (larger $T$) makes it possible to reach
a more optimal value of $\Sinf$.

\begin{figure}
    \centering
    \subfigure[Solution and control vs. time]{\label{fig:ex1-time}\includegraphics[width=0.65\textwidth]{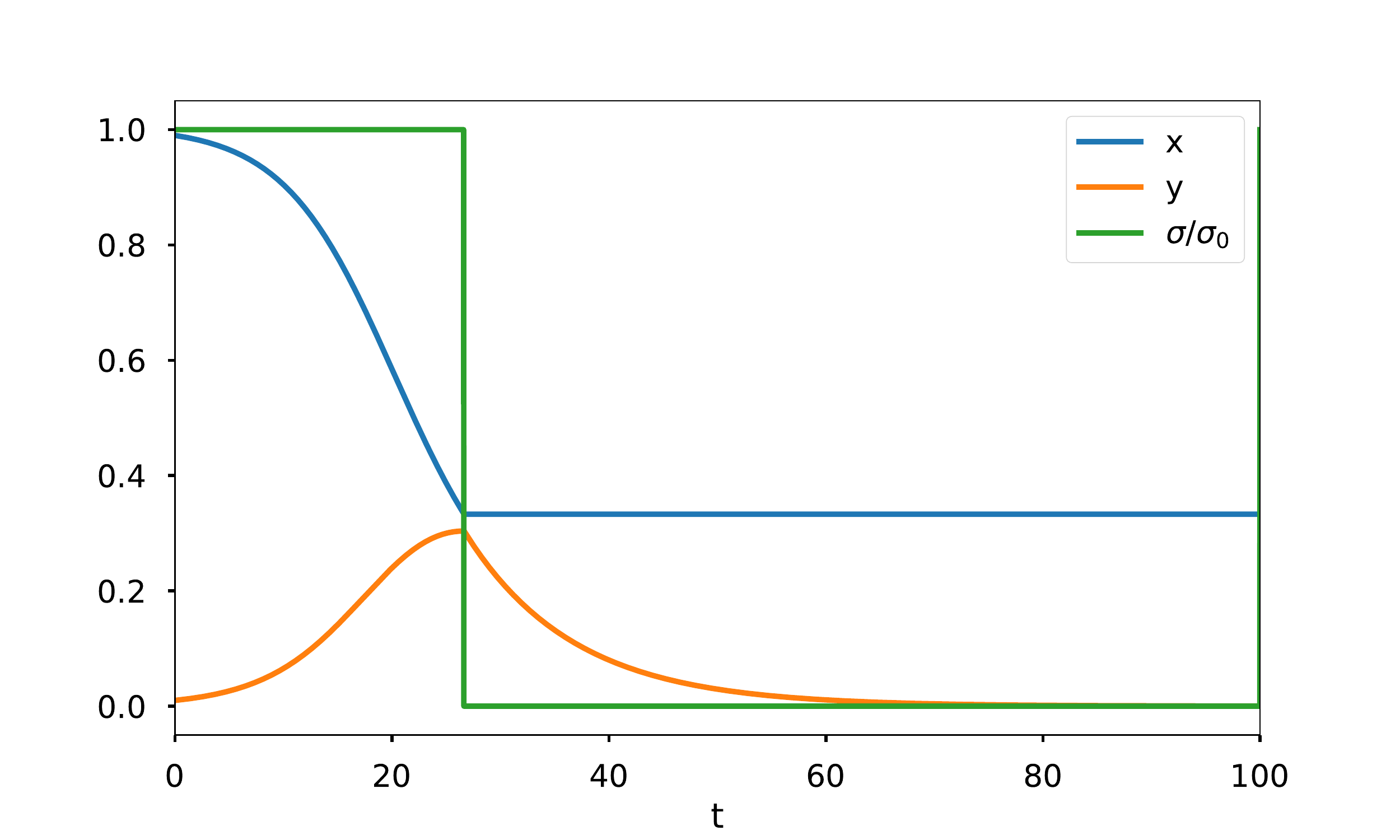}}
    \subfigure[Trajectory in phase space]{\label{fig:ex1-xy}\includegraphics[width=0.34\textwidth]{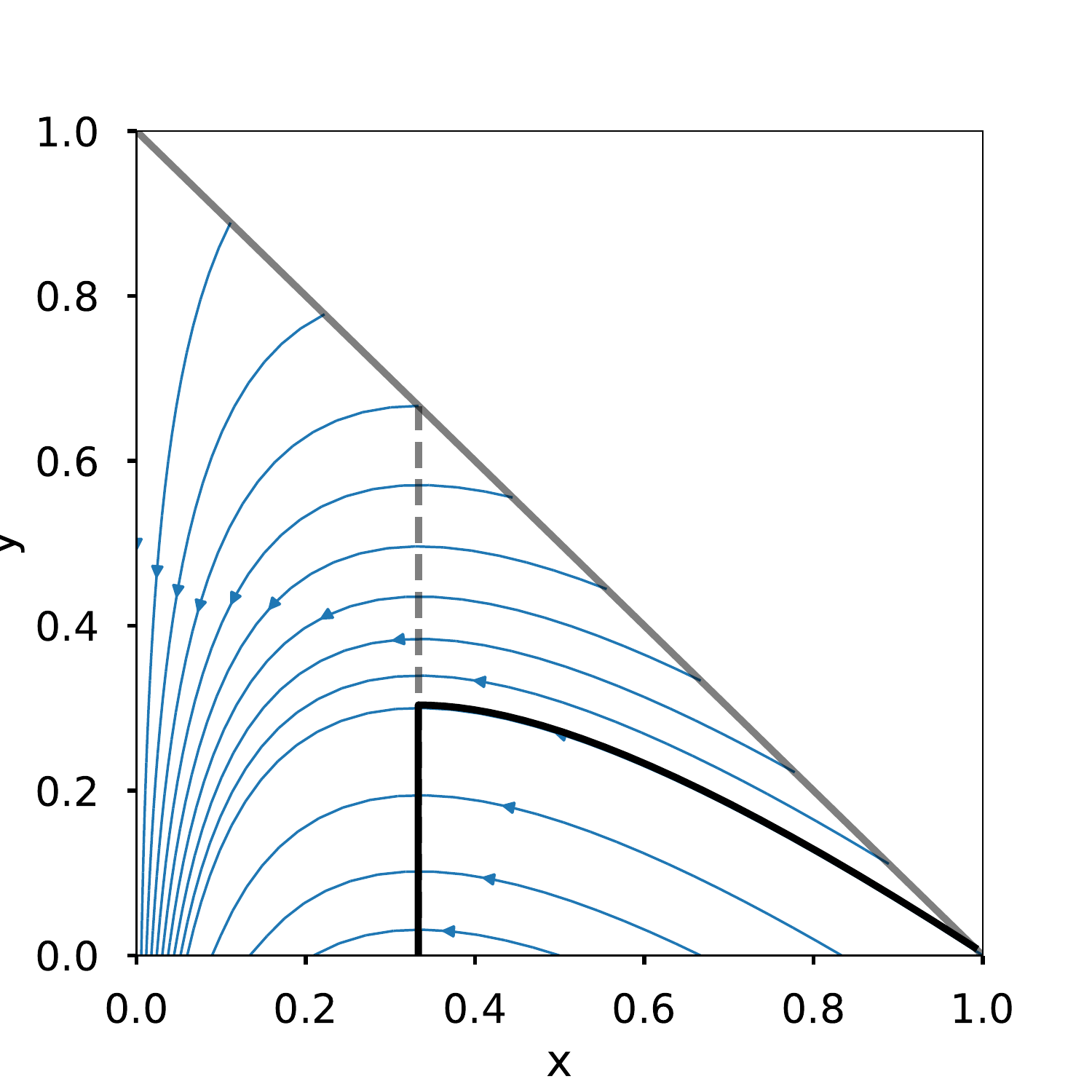}}
    \caption{Typical optimal solution.  Here $(x(0),y(0)) = (0.99,0.01)$, $\beta=0.3$, and $\gamma=0.1$.\label{fig:example1}}
\end{figure}

\begin{figure}
    \centering
    \includegraphics[width=0.5\textwidth]{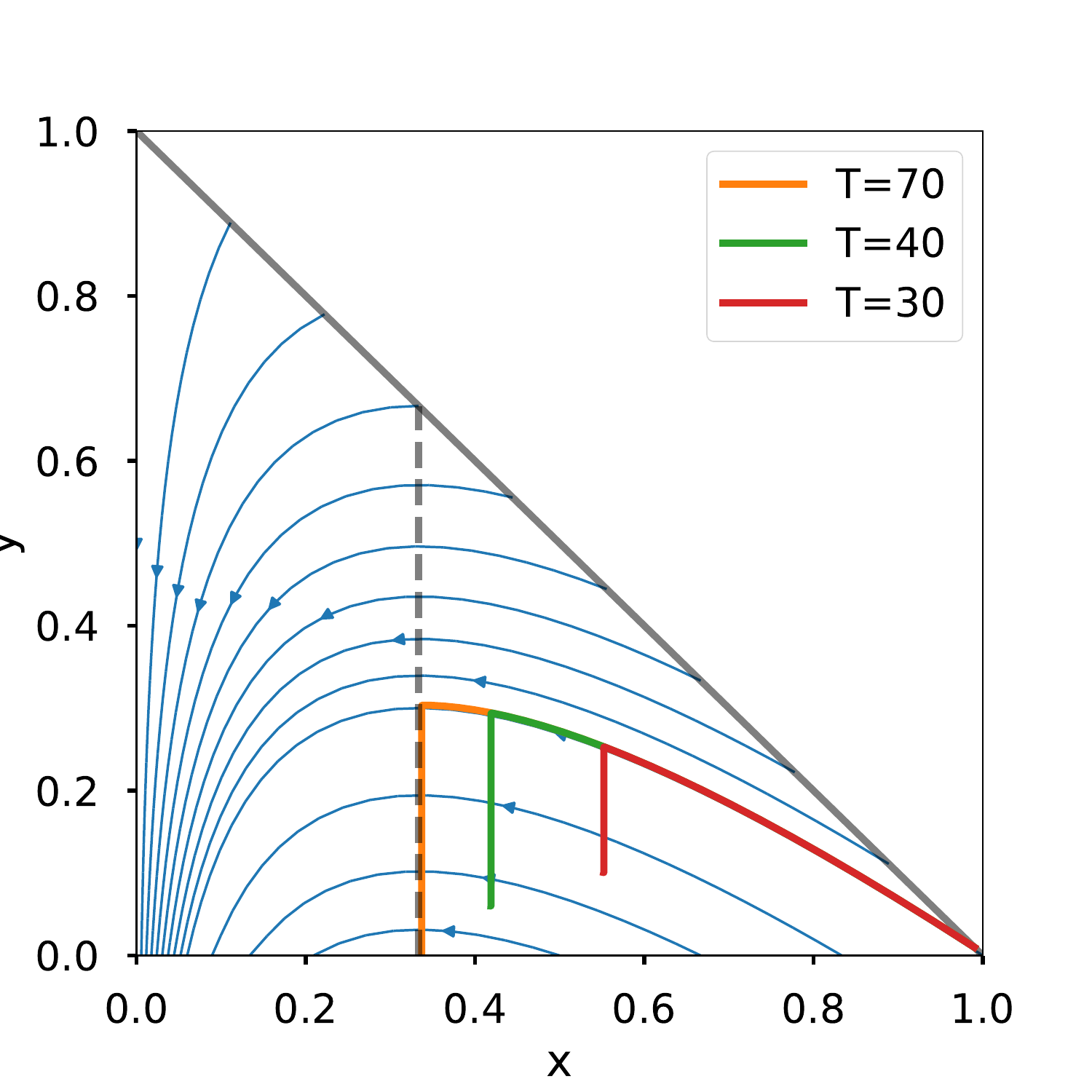}
    \caption{Optimal solutions starting from the same point $(0.99,0.01)$ but with different
        final times.  A larger value of $T$ allows the system to reach a more optimal state.
        For all solutions, $\beta=0.3$ and $\gamma=0.1$.\label{fig:diff-time-opt}}
\end{figure}

In real-world scenarios, it may not be possible to apply the maximum control $\sigma(t)=0$.
Suppose that in place of \eqref{q-shortterm} we impose $\sigma_\textup{min} \le \sigma(t) \le \Rnot$.
In this case the optimal control is still bang-bang with a single switching time.
In Figure \ref{fig:example_2}, we show an optimal solution when $\sigma(t)\ge 0.4\Rnot$ is imposed.

\begin{figure}
    \centering
    \subfigure[Solution and control vs. time]{\label{fig:example_2-t}\includegraphics[width=0.65\textwidth]{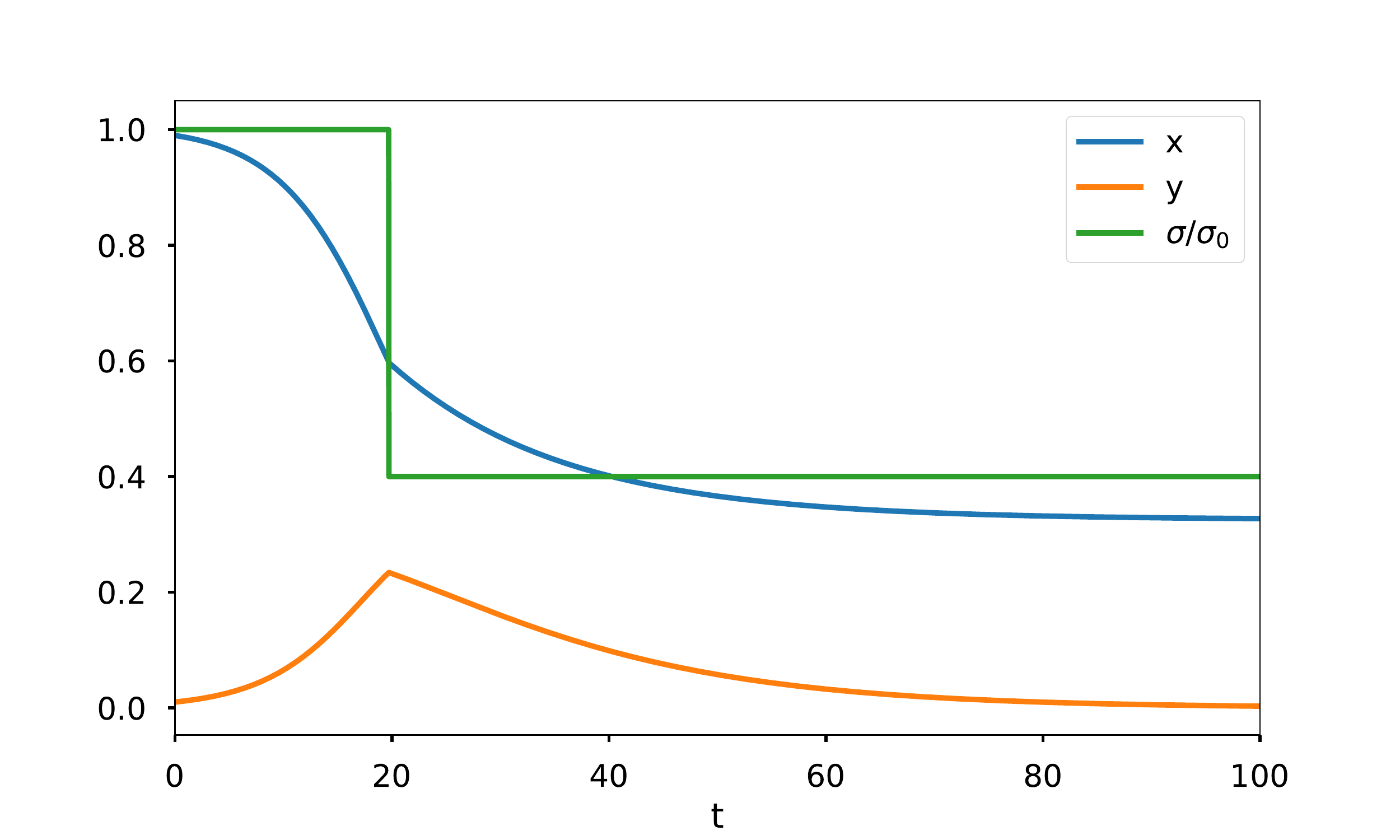}}
    \subfigure[Trajectory in phase space]{\label{fig:example_2-xy}\includegraphics[width=0.34\textwidth]{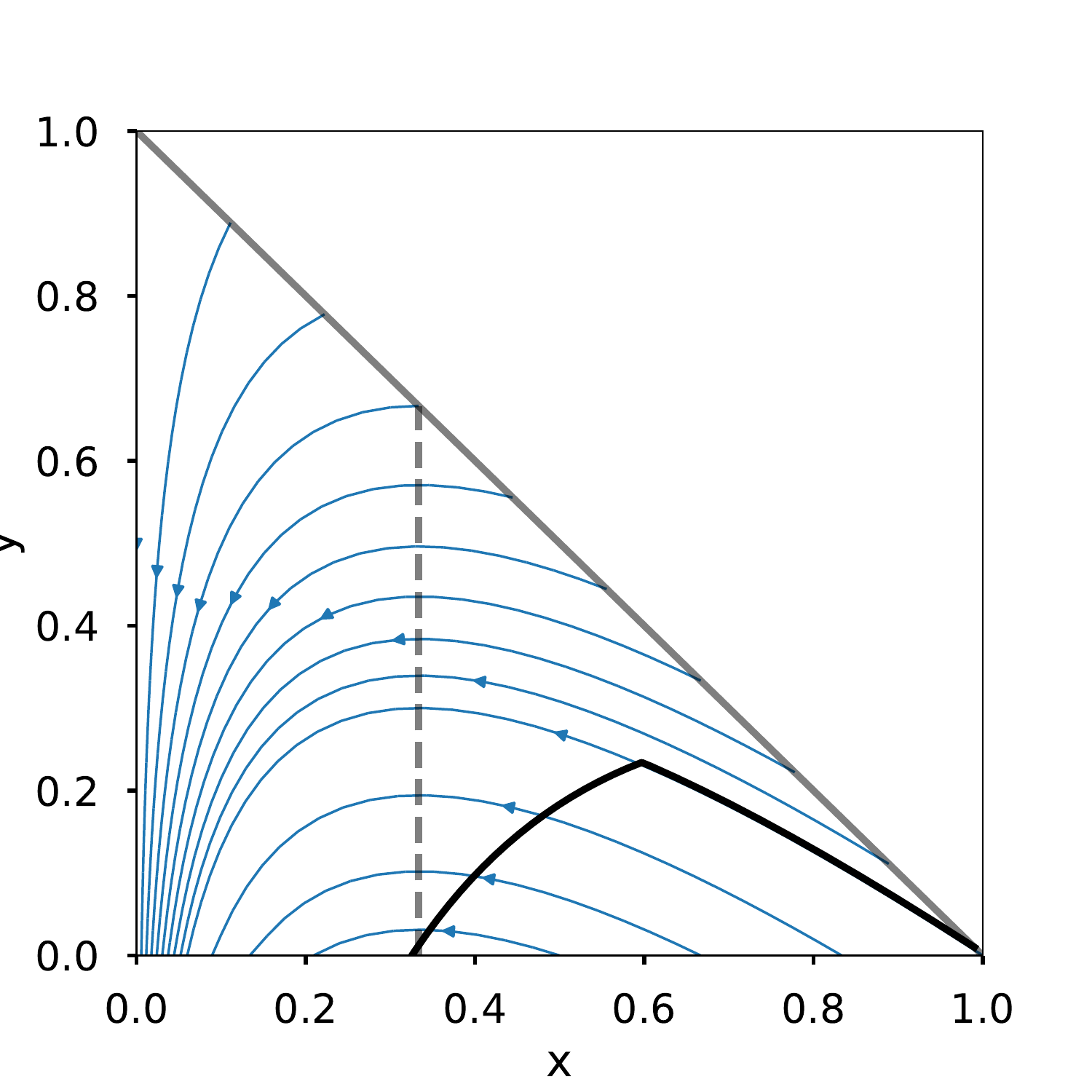}}
    \caption{Optimal solutions with $\sigma(t)\ge 0.4\Rnot$.  Here $(x(0),y(0)) = (0.99,0.01)$, $\beta=0.3$, $\gamma=0.1$, and $T=100$.\label{fig:example_2}}
\end{figure}

The result above can also be obtained via the Hamilton-Jacobi-Bellman (HJB) equation
for \eqref{eq:no-cost-problem}.  Here we sketch this approach.  The HJB
equation for $u(x,y,t)$ can be written
\begin{subequations} \label{eq:hjb-no-cost}
\begin{align}
    u_t & = \gamma y u_y - \gamma x y \min_\sigma \left((u_y-u_x)\sigma\right) \\
    u(x,y,T) & = -\Sinf(x,y,\Rnot).
\end{align}
\end{subequations}
The required minimum is obtained by taking
\begin{align} \label{eq:hjb-control}
    \sigma(t) & = \begin{cases} 
        0 & u_y(x,y,t) > u_x(x,y,t) \\
        \Rnot & u_y(x,y,t)<u_x(x,y,t).
    \end{cases}
\end{align}
From \eqref{xinf-grad} we see that $u_y(x,y,T)>u_x(x,y,T)$ for all $(x,y)$.
Thus for small enough values of $T-t$, the solution of \eqref{eq:hjb-no-cost} satisfies
\begin{align*}
    u_t & = \gamma y u_y(x,y,t).
\end{align*}
The solution of this hyperbolic PDE is
\begin{align*}
    u(x,y,t) & = u(x,ye^{-\gamma(T-t)},T) = -\Sinf(x,ye^{-\gamma(T-t)}).
\end{align*}
Thus, for small enough $T-t$,
\begin{align*}
    u_x(x,y,t) & = -\frac{\partial \Sinf}{\partial y} \left(1- \frac{1}{x(t)\Rnot}\right) \\
    u_y(x,y,t) & = -\frac{\partial \Sinf}{\partial y} e^{-\gamma(T-t)}.
\end{align*}
According to \eqref{eq:hjb-control}, the optimal control value will switch when $u_x=u_y$, which
leads to \eqref{xtstar}.  Meanwhile, substituting \eqref{eq:hjb-control} in \eqref{eq:hjb-no-cost}
in the case $u_y<u_x$ yields the linear hyperbolic PDE
$$
    u_t = \gamma y u_y - \beta x y (u_y-u_x),
$$
whose characteristics are just the trajectories of the SIR system
\eqref{SIR} illustrated in Figure \ref{fig:dynamics}, which are also
contours of $\Sinf$.  It can be shown that once $u_y-u_x<0$, this inequality
will continue to hold along each such characteristic.

\section{Optimal control with $L\ne 0$\label{sec:exploration}}
We now consider the case of a non-zero Lagrangian, which allows us to account for factors
like the economic cost of intervention or heightened risks caused by hospital
overflow.  We formulate the Hamiltion-Jacobi-Bellman (HJB) equation for this problem and
apply an upwind numerical method to compute approximate solutions.
The numerical solutions obtained via the HJB equation have also been checked in each case
against solutions of the BVP given in Section \ref{sec:pmp}, and found to
agree within numerical errors.

Because the Lagrangian in this section is not a linear function, the
solution is not bang-bang, and instead varies smoothly (except when it
reaches the minimum or maximum allowable value).

\subsection{Quadratic running cost of control}
We now attempt to account for the economic cost of intervention.  Quantification
of the cost of measures like closing schools and businesses is a challenging
problem in economic modeling, and well outside the scope of the present work.
Based on the general idea that both the cost and the marginal cost will increase
with the degree of contact reduction, we take for simplicity
$$
    L(x(t),y(t),\sigma(t)) = c_2 \left(1-\frac{\sigma(t)}{\Rnot}\right)^2.
$$
The HJB equation for \eqref{eq:basic-problem} is then
\begin{subequations} \label{eq:hjb-q-cost}
\begin{align} \label{eq:hjb-q-cost-pde}
    u_t - \gamma y u_y & = - \min_{0\le \sigma\le \Rnot} \left((u_y-u_x)\gamma x y \sigma(t) + c_2 \left(1-\frac{\sigma(t)}{\Rnot}\right)^2 \right) \\
    u(x,y,T) & = -\Sinf(x,y,\Rnot).
\end{align}
\end{subequations}
The minimum in \eqref{eq:hjb-q-cost-pde} is obtained with
\begin{align} \label{eq:hjb-sigma}
    \sigma(t) & = \Rnot\min\left(1,\max\left(0,\left(1-\frac{\Rnot\gamma}{2c_2}x y (u_y-u_x)\right)\right)\right).
\end{align}
We approximate the solution of \eqref{eq:hjb-q-cost}-\eqref{eq:hjb-sigma} using a second-order
finite volume discretization with the PyClaw software \cite{2012_pyclaw-sisc,2013_sharpclaw,2016_clawpack}.
For details, the reader is referred to the reproducibility repository that contains the code
for all examples in this paper \cite{ketcheson2021SIRRepro}.

Numerical solutions for a range of values of $c_2$ are shown in Figure \ref{fig:varying_c2}.
The values of $c_2$ used here are chosen merely to illustrate the range of possible behaviors.
Notice that the strength of the control $\sigma(t)$ and the number of infected at certain times vary non-monotonically
with $c_2$.  Indeed, the optimal
control $\sigma(t)$ up to around day 15 is simply $\Rnot$ in both limits $c_2 \to \infty$ and $c_2 \to 0$,
whereas for intermediate values of $c_2$ some intervention is imposed in this period.

\begin{figure}
    \centering
    \subfigure[Solution (solid line) and control $\sigma(t)/\Rnot$ (dashed line) vs. time]{\label{fig:varying-c2}\includegraphics[width=0.65\textwidth]{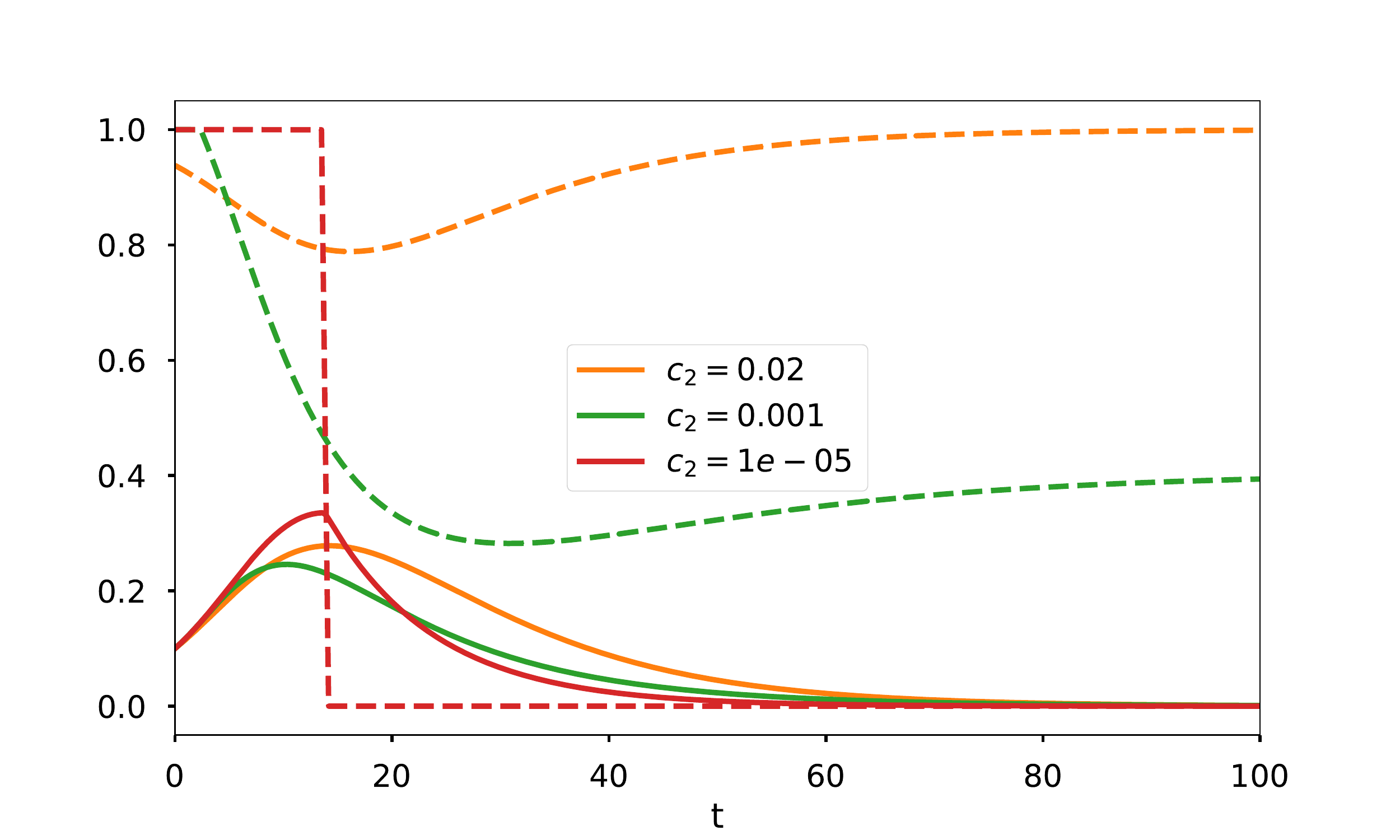}}
    \subfigure[Trajectory in phase space]{\label{fig:varying-cs-xy}\includegraphics[width=0.34\textwidth]{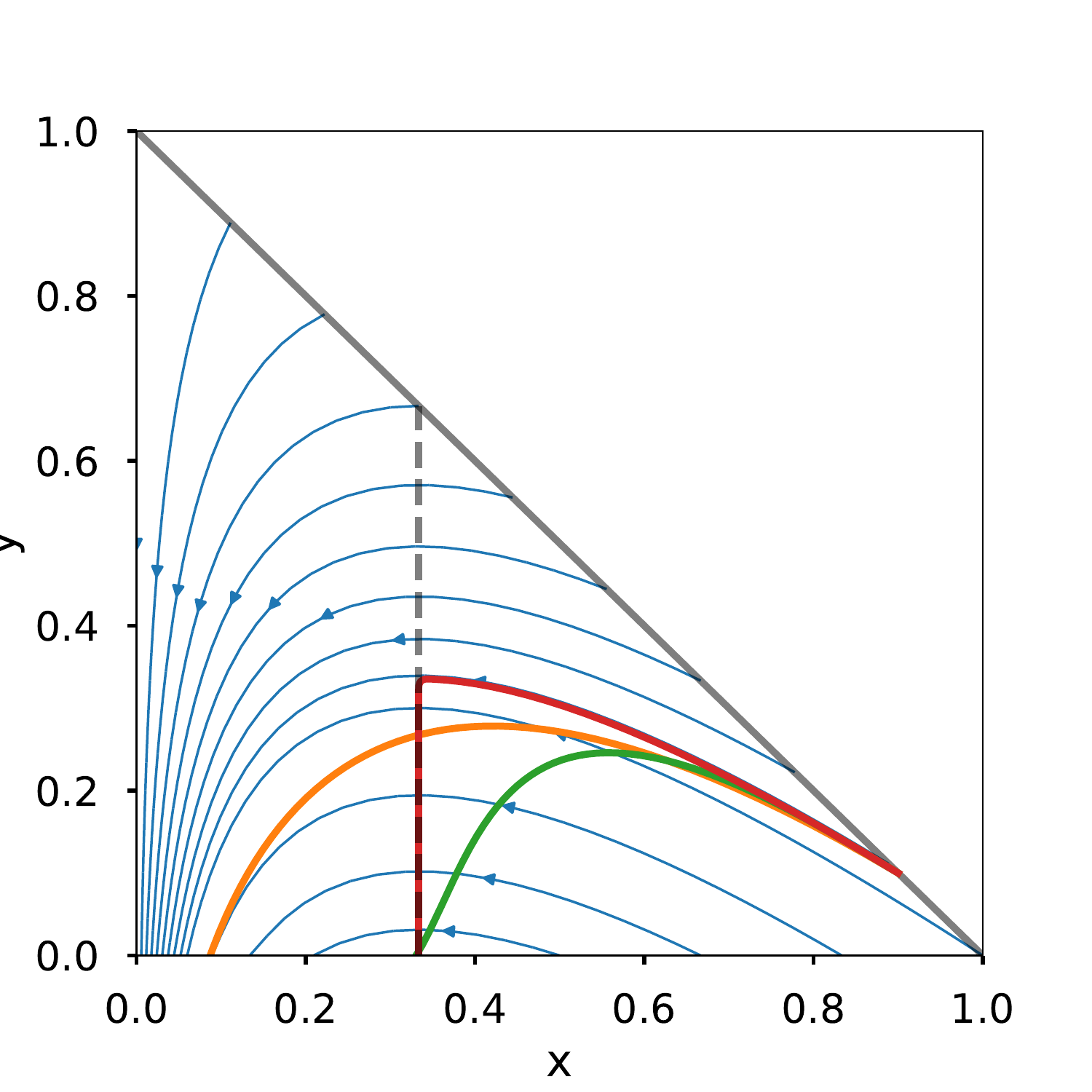}}
    \caption{Optimal solutions with different running cost.  Here $(x(0),y(0)) = (0.9,0.1)$, $\beta=0.3$, $\gamma=0.1$, and $T=100$.\label{fig:varying_c2}}
\end{figure}

\subsection{Minimizing hospital overflow}
The optimal solutions above may be unsatisfactory in practice,
since the number of people simultaneously infected at certain times
may be too great for all of them to receive adequate medical care.  This is a major concern
with respect to the current COVID-19 crisis.  A natural objective is to keep
the number of infected below some threshold, corresponding for instance to the
number of hospital beds.  We thus consider the Lagrangian
$$
    L(x(t),y(t),\sigma(t)) = c_2 \left(1-\frac{\sigma(t)}{\Rnot}\right)^2 + c_3 g(y(t)-y_\text{max}).
$$
Here $\ymax$ is the maximum number of hospital beds.
The HJB equation is then
\begin{subequations} \label{eq:hjb-hosp}
\begin{align} \label{eq:hjb-hosp-pde}
    u_t - \gamma y u_y + c_3 g(y-y_\text{max}) & = - \min_{0\le \sigma\le \Rnot} \left((u_y-u_x)\gamma x y \sigma(t) + c_2 \left(1-\frac{\sigma(t)}{\Rnot}\right)^2 \right) \\
    u(x,y,T) & = -\Sinf(x,y,\Rnot).
\end{align}
\end{subequations}
The control that achieves the minimum in \eqref{eq:hjb-hosp-pde} is again given by \eqref{eq:hjb-sigma}.
The function $g(v)$ should be nearly zero for $v<0$ and increase
in an approximately linear fashion for $v>0$.  For the purpose of having
a tractable control problem, it is also desirable that $g$ be differentiable.
We take
$$
g(v) = \frac{v}{1+e^{-100v}}.
$$
Figures \ref{fig:min_hosp_1} and \ref{fig:min_hosp_2} show examples of solutions.
Again, we choose parameter values that demonstrate the range of qualitative behaviors.
In both examples, the cost of control is scaled by $c_2=10^{-2}$.
In Figure \ref{fig:min_hosp_1}, a higher cost for hospital overflow is applied,
with $c_3=100$.  As might be expected, $y(t)$ is generally
kept below $\ymax$ (which is set to $0.1$).
The control is initially off, then turns on
to avoid hospital overflow, and then turns off again.  While the control is applied,
it is maintained at a level that keeps the value of $y(t)$ nearly constant in time.

\begin{figure}
    \centering
    \subfigure[Solution and control vs. time]{\label{fig:minhosp1-time}\includegraphics[width=0.65\textwidth]{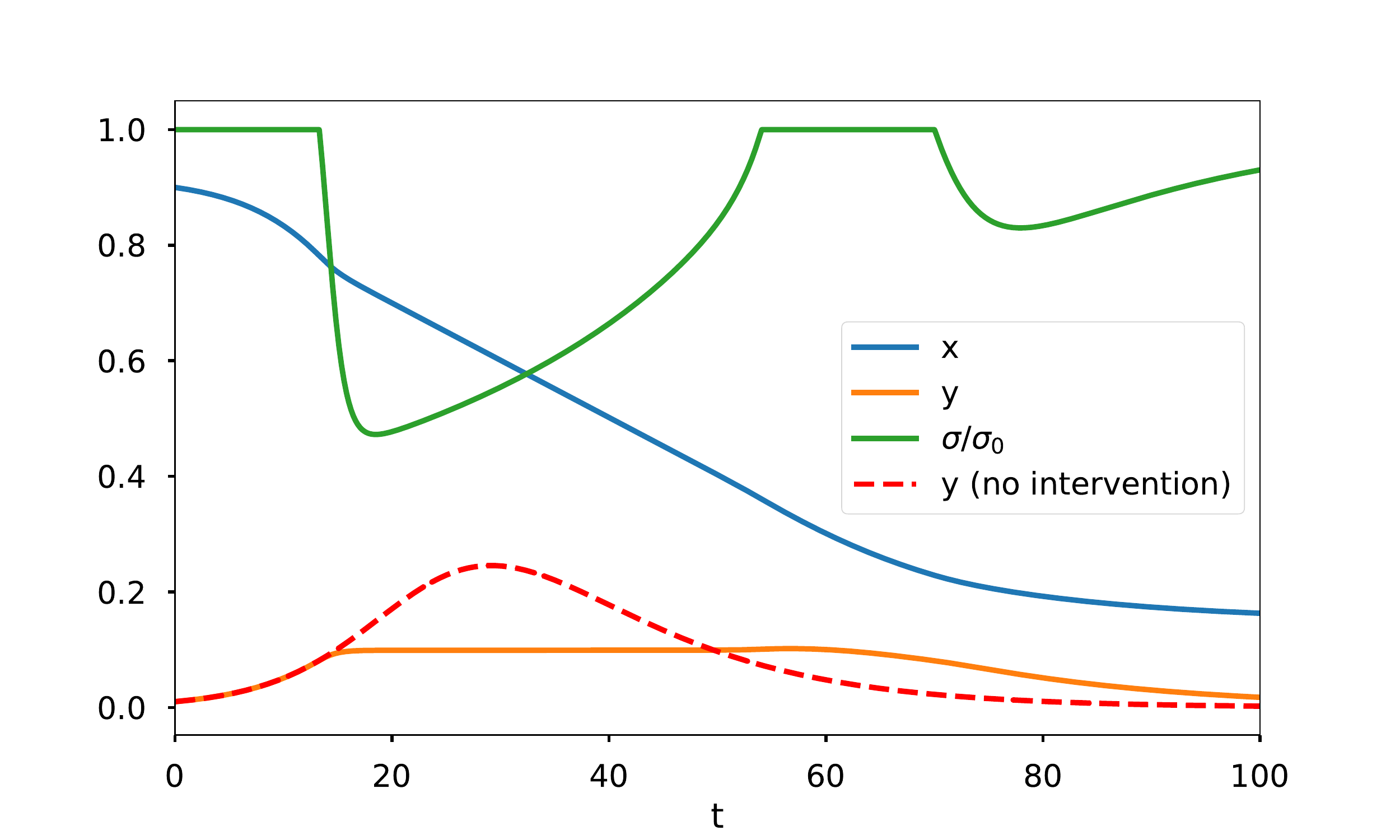}}
    \subfigure[Trajectory in phase space]{\label{fig:minhosp1-xy}\includegraphics[width=0.34\textwidth]{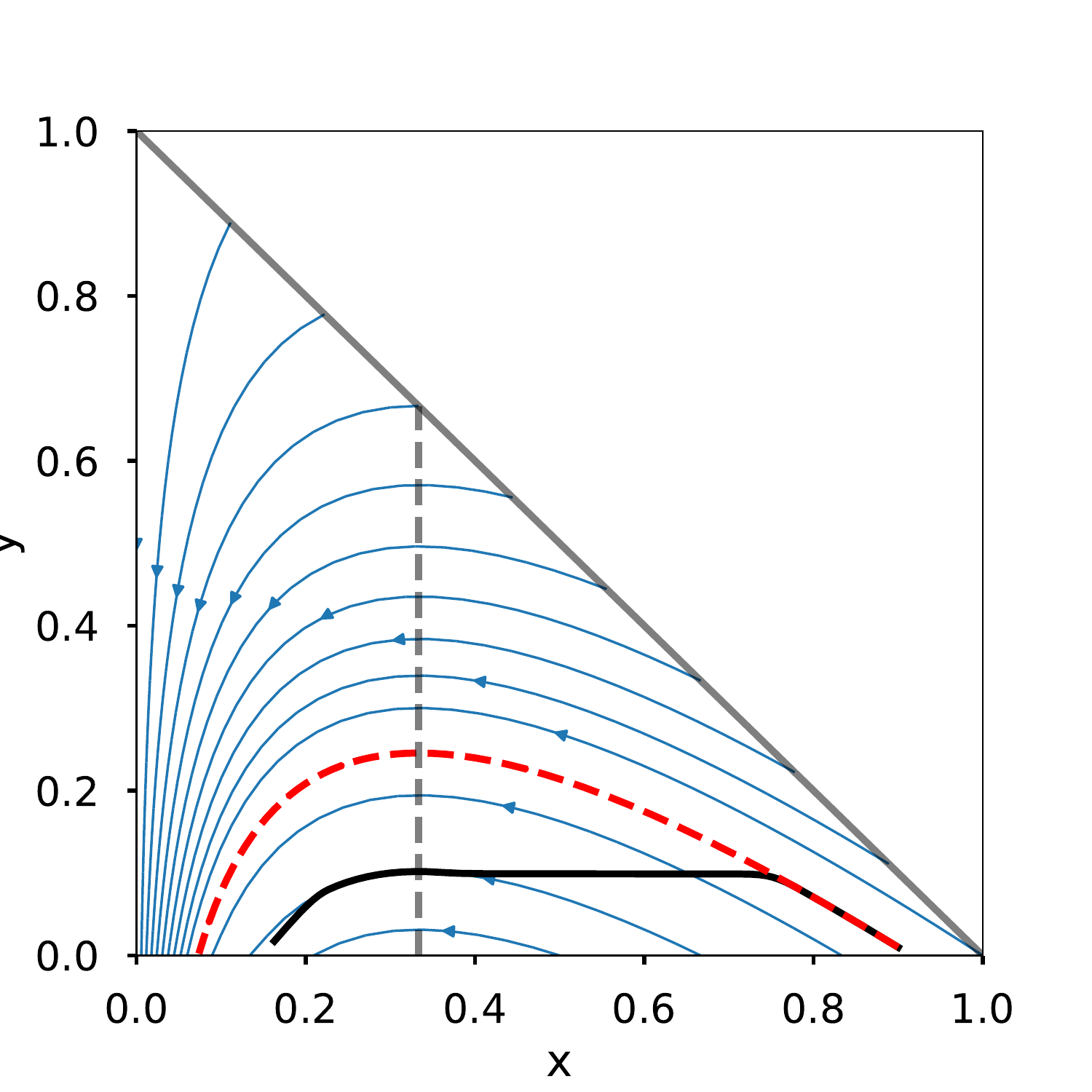}}
    \caption{Optimal solutions with cost for hospital overflow.  Here $(x(0),y(0)) = (0.9,0.01)$, $\beta=0.3$, $\gamma=0.1$, $T=100$,
        and $\ymax=0.1$.
        In the cost function, we take $c_2=10^{-2}$ and $c_3=100$.
        The dashed red line shows the result of imposing no control.
        \label{fig:min_hosp_1}}
\end{figure}

Figure \ref{fig:min_hosp_2} shows another example scenario in which the cost of hospital overflow is smaller,
with $c_3=1$.  In this case the hospital capacity is significantly exceeded for a short time,
and the control is kept on until the final time, but the epidemiological overshoot
is significantly reduced compared to the previous solution.

\begin{figure}
    \centering
    \subfigure[Solution and control vs. time]{\label{fig:minhosp2-time}\includegraphics[width=0.65\textwidth]{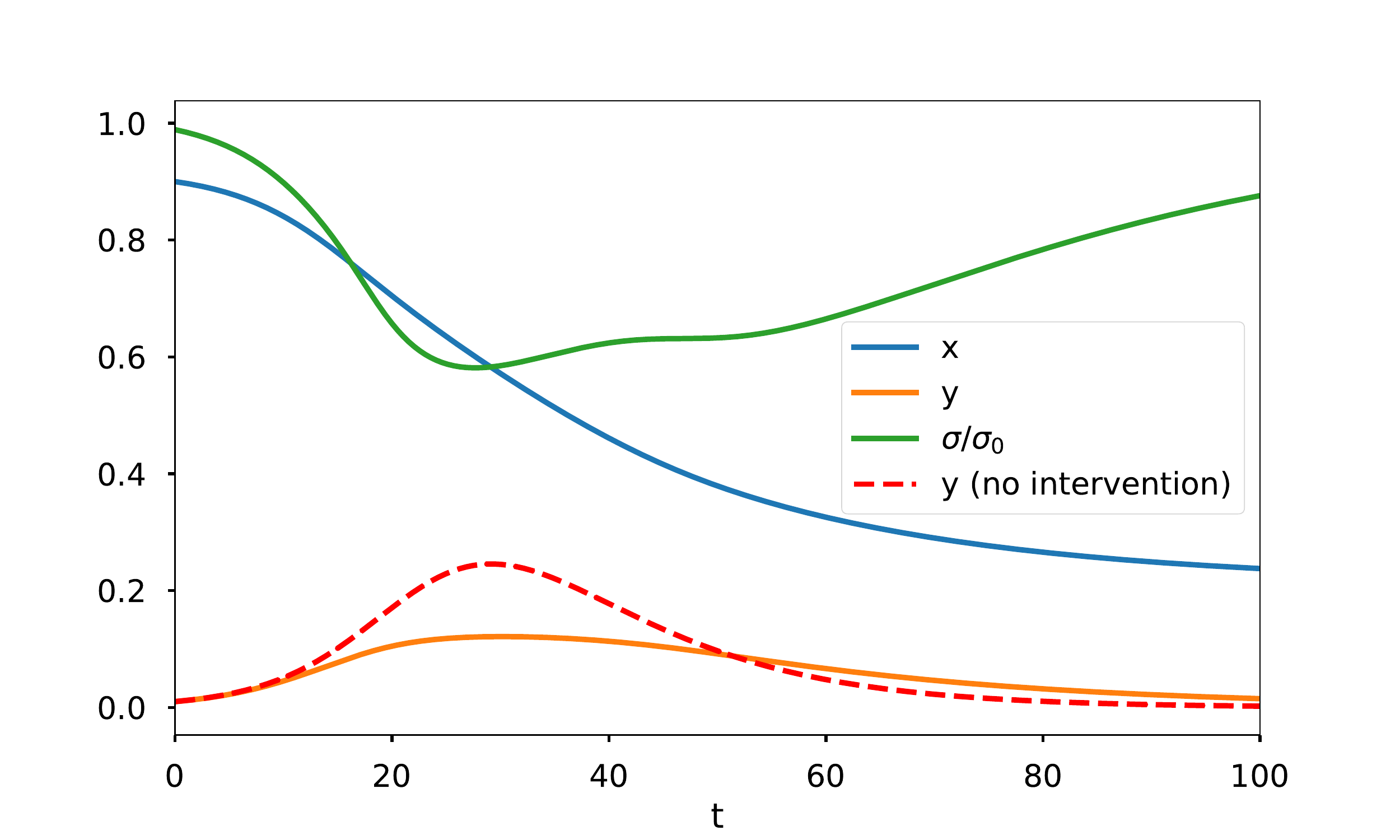}}
    \subfigure[Trajectory in phase space]{\label{fig:minhosp2-xy}\includegraphics[width=0.34\textwidth]{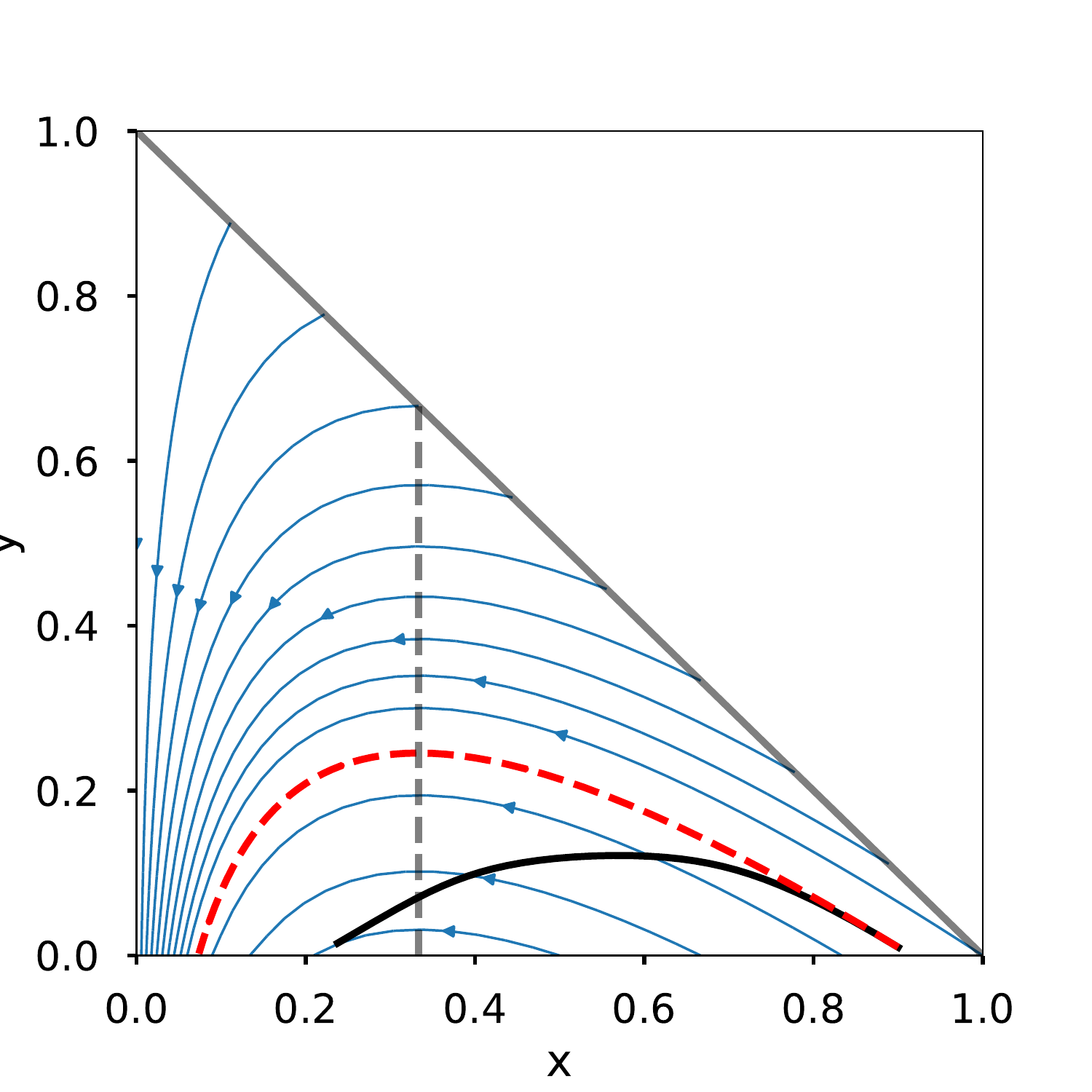}}
    \caption{Optimal solutions with smaller cost for hospital overflow.  Here $(x(0),y(0)) = (0.9,0.01)$, $\beta=0.3$, $\gamma=0.1$, $T=100$,
        and $\ymax=0.1$.
        In the cost function, we take $c_2=10^{-2}$ and $c_3=1$.
        The dashed red line shows the result of imposing no control.
        \label{fig:min_hosp_2}}
\end{figure}

\section{Application to the COVID-19 pandemic\label{sec:application}}
The main goal of this work has been a mathematical investigation of optimal
controls for the SIR model with a controlled rate of contact, as presented
in the previous sections.  We now present a brief illustration of the results
in practical terms through application to the current COVID-19 pandemic.
This application is imprecise, for several reasons:  the SIR model
is one of the simplest epidemiological models, and assumes homogeneous mixing
among a population; the current state of susceptible and infected persons is
not accurately known; and the parameters of the disease itself (i.e. $\gamma, \Rnot$)
are still quite uncertain.  
The examples in this section should be viewed only as illustrations
of a few possible scenarios, and not an exhaustive or detailed study.

We take the infectious period $\gamma^{-1}=10$ days, and the basic reproduction
number $\Rnot=3.2$, based on recent estimates \cite{verity2020estimates,liu2020reproductive}.
To make the results easy to interpret, we use a
fixed terminal cost of $c_1 z_\infty$, where we have introduced an additional
scaling constant.  Taking $c_1 = \alpha N$, where $N$ is the total population
being modeled and $\alpha$ is the infection fatality ratio, then this cost is
the expected number of lives lost.  Since $z_\infty=1-\Sinf$, this is merely a
rescaling of the terminal cost used throughout this work.  We take
$\alpha \approx 0.006$ based on recent estimates \cite{verity2020estimates,russell2020estimating,wu2020estimating}.

We seek reasonable order-of-magnitude estimates for $c_2$ and $c_3$.  
The value of $c_3/N$ should be equal to the increase in probability of a given
infected person dying because of the lack of medical care.  We take
$c_3 = N\eta$,
where the fatality ratio in the absence of medical care is $\alpha+\eta$.
We take $\eta\approx \alpha$, giving
$c_3 = 0.006$.  For $\ymax$ we take values from the United States, where
there are about 3 hospital beds per 1000 people, and two-thirds of them
are typically occupied.  Since it is estimated that about 5\% of COVID-19
cases are hospitalized \cite{verity2020estimates}, this gives $\ymax=0.02N$.

Any attempt to quantify the cost of an intervention in human lives is bound
to be contentious.  Whether we consider the value of a human life to be in
intrinsic personal value or extrinsic economic value, we can view the cost
of intervention as a reduction of the value of human lives during the intervention
period.  We take $c_2 = N\epsilon/d$ where $d\approx 10^4$ is the number
of days in a human life (more precisely, the average number of days remaining
in a life claimed by the disease) and $1-\epsilon$ is the relative value of a
day spent in full isolation ($\sigma=0$) compared to a day without intervention.
Taking $\epsilon=0.2$, we have $c_2 = 2\times10^{-5} N$.

Since all terms in the cost function are proportional to $N$, we take $N=1$ without
loss of generality.  Results for the parameter values given above are shown in Figure
\ref{fig:real-world-1}.  We see that the optimal control corresponds to a level
of intervention that becomes more strict as the epidemic grows, and is gradually
relaxed as the epidemic subsides.  Most importantly, and in agreement with results
from the examples in earlier sections, the strongest control is applied around
the time of peak infection and shortly thereafter.
The resulting epidemiological overshoot is very small.

\begin{figure}
    \centering
    \subfigure[Solution and control vs. time]{\label{fig:real_world_1-time}\includegraphics[width=0.65\textwidth]{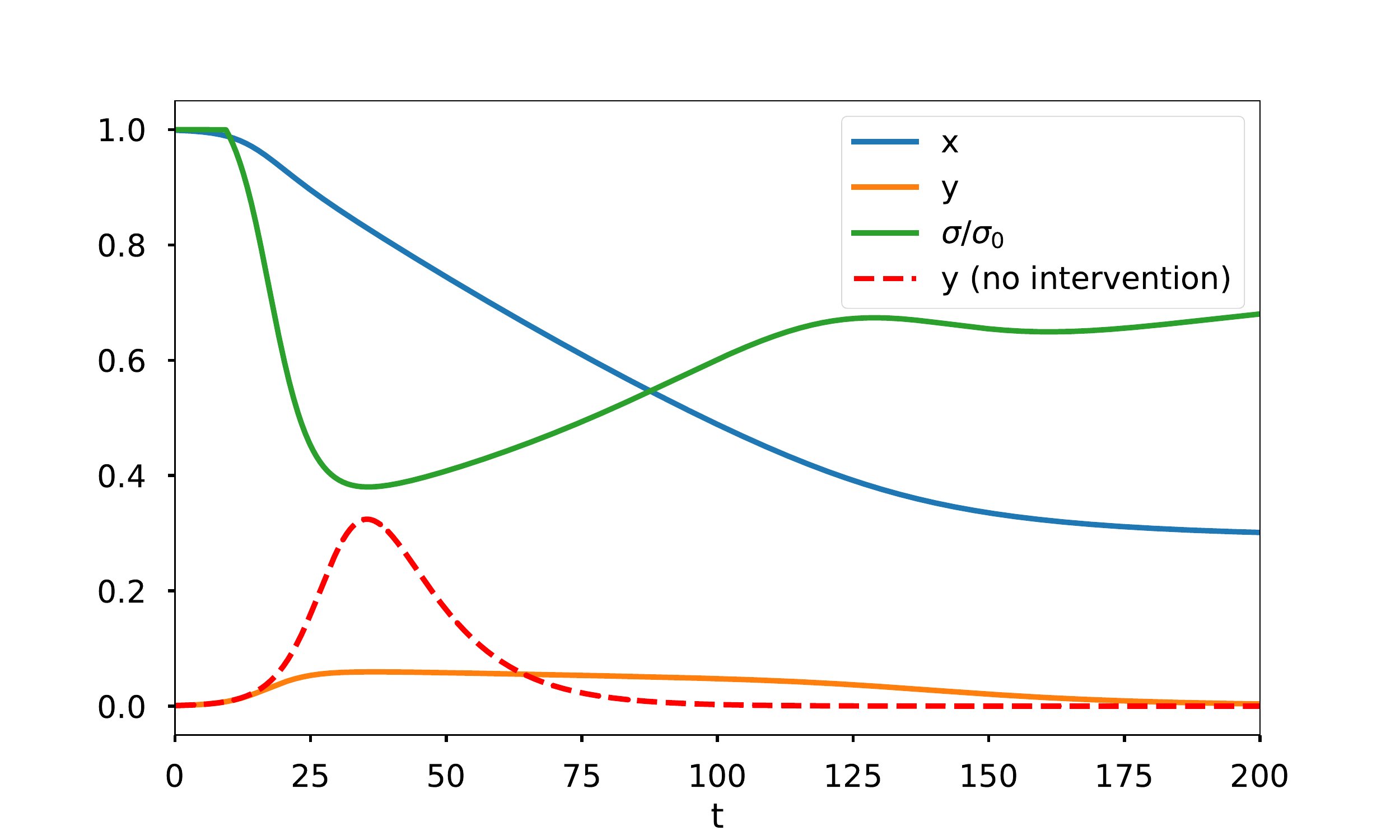}}
    \subfigure[Trajectory in phase space]{\label{fig:real_world_1-xy}\includegraphics[width=0.34\textwidth]{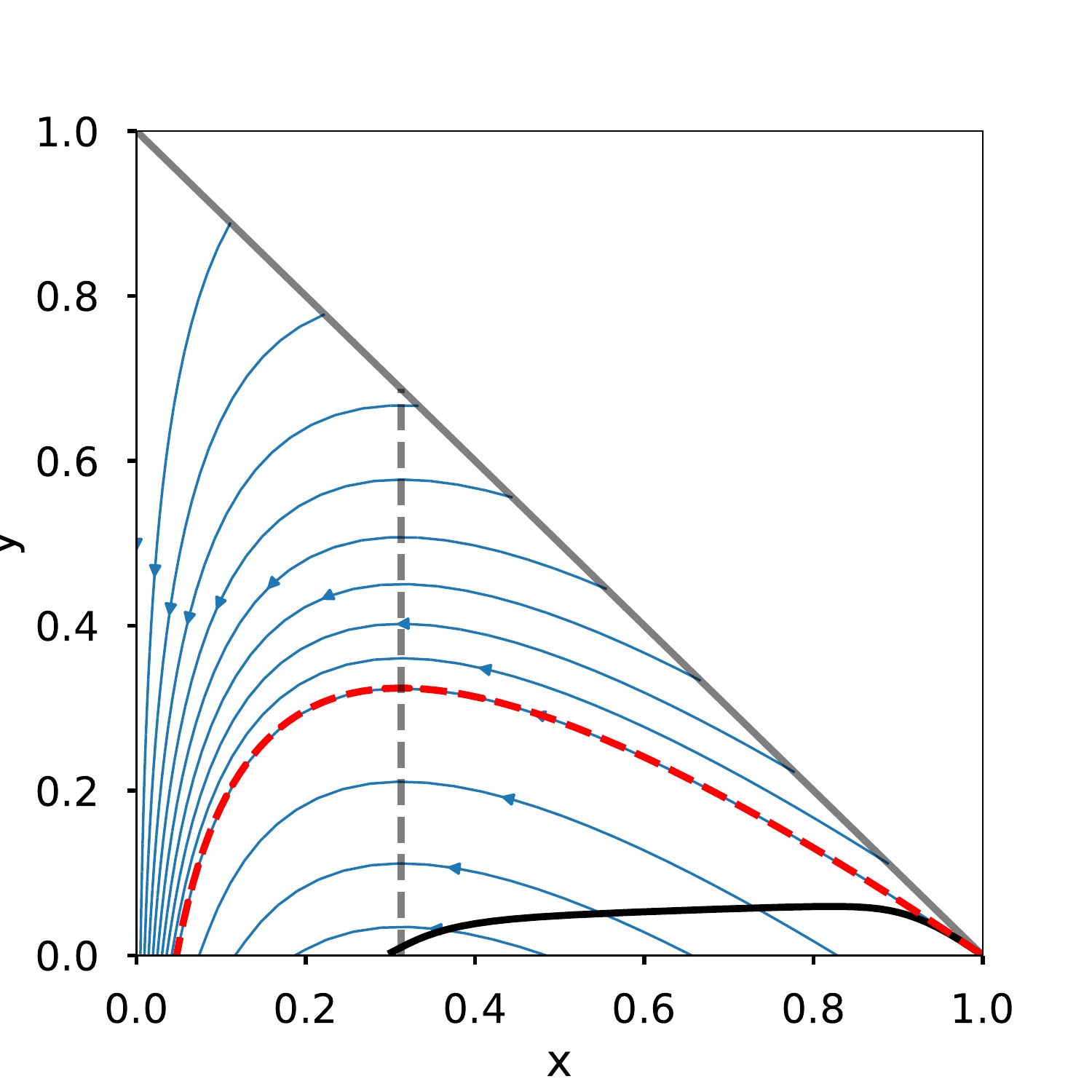}}
    \caption{Optimal control for COVID-19 with $\Rnot=3.2$, $\gamma=0.1$,
                $\alpha=\eta=0.006$, $\epsilon=0.2$, $d=10^4$, $T=200$, and $(x(0),y(0)) =
                (0.999,0.001)$.
                The dashed red line shows the result of imposing no control.
                \label{fig:real-world-1}}
\end{figure}

An alternative scenario is shown in Figure \ref{fig:real-world-2}, in which
we have assumed a fatality ratio and a value of $\eta$ that are twice as
large (in line with the highest estimates of the infection fatality ratio),
as well as taking a smaller cost of intervention with $\epsilon=0.05$.  These
parameters lead to stronger intervention, especially in the later phases of the epidemic.
The result is almost no epidemiological overshoot.

\begin{figure}
    \centering
    \subfigure[Solution and control vs. time]{\label{fig:real_world_2-time}\includegraphics[width=0.65\textwidth]{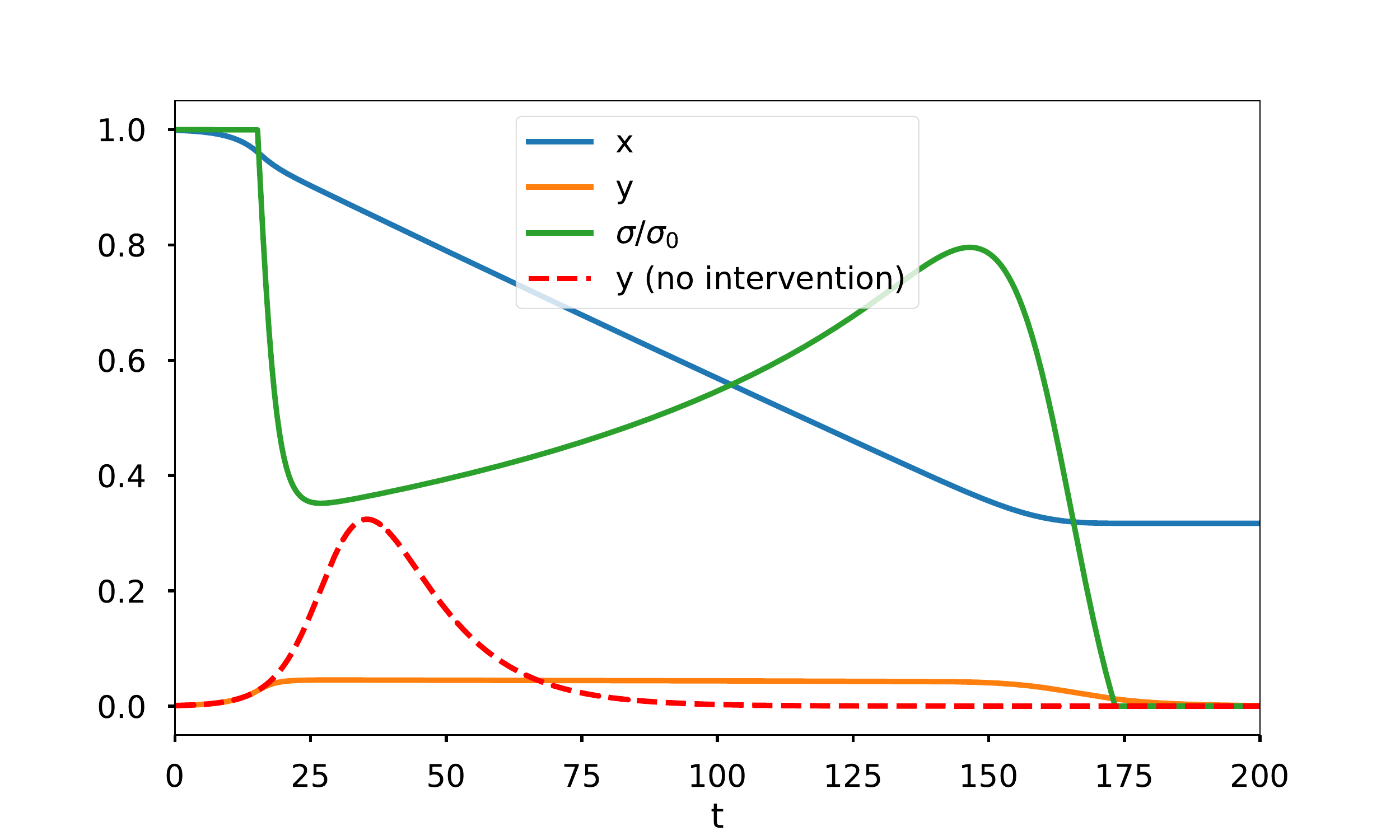}}
    \subfigure[Trajectory in phase space]{\label{fig:real_world_2-xy}\includegraphics[width=0.34\textwidth]{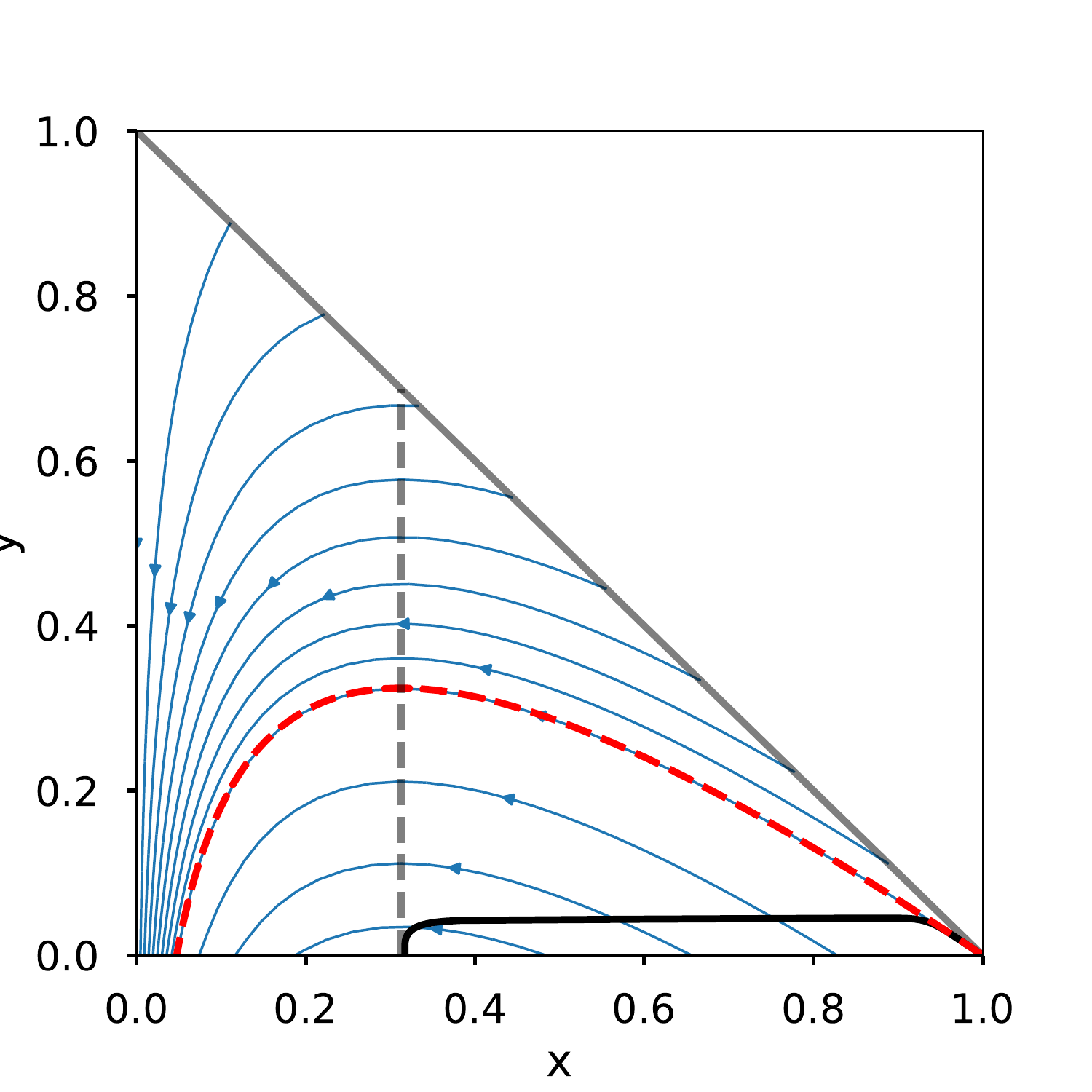}}
    \caption{Optimal control for COVID-19 with $\Rnot=3.2$, $\gamma=0.1$,
                $\alpha=\eta=0.012$, $\epsilon=0.05$, $d=10^4$, $T=200$, and $(x(0),y(0)) =
                (0.999,0.001)$.
                The dashed red line shows the result of imposing no control.
                \label{fig:real-world-2}}
\end{figure}

Finally, in Figure \ref{fig:real-world-3}, we repeat the first scenario
but increase the cost of control by taking $\epsilon=1$.
In this case a more mild control is applied, peaking at about
35\% contact reduction and concentrated around the time of the
infection peak.  In this case the optimal solution includes a small
but significant epidemiological overshoot, and significantly exceeds
the available hospital beds for a certain period of time.

\begin{figure}
    \centering
    \subfigure[Solution and control vs. time]{\label{fig:real_world_3-time}\includegraphics[width=0.65\textwidth]{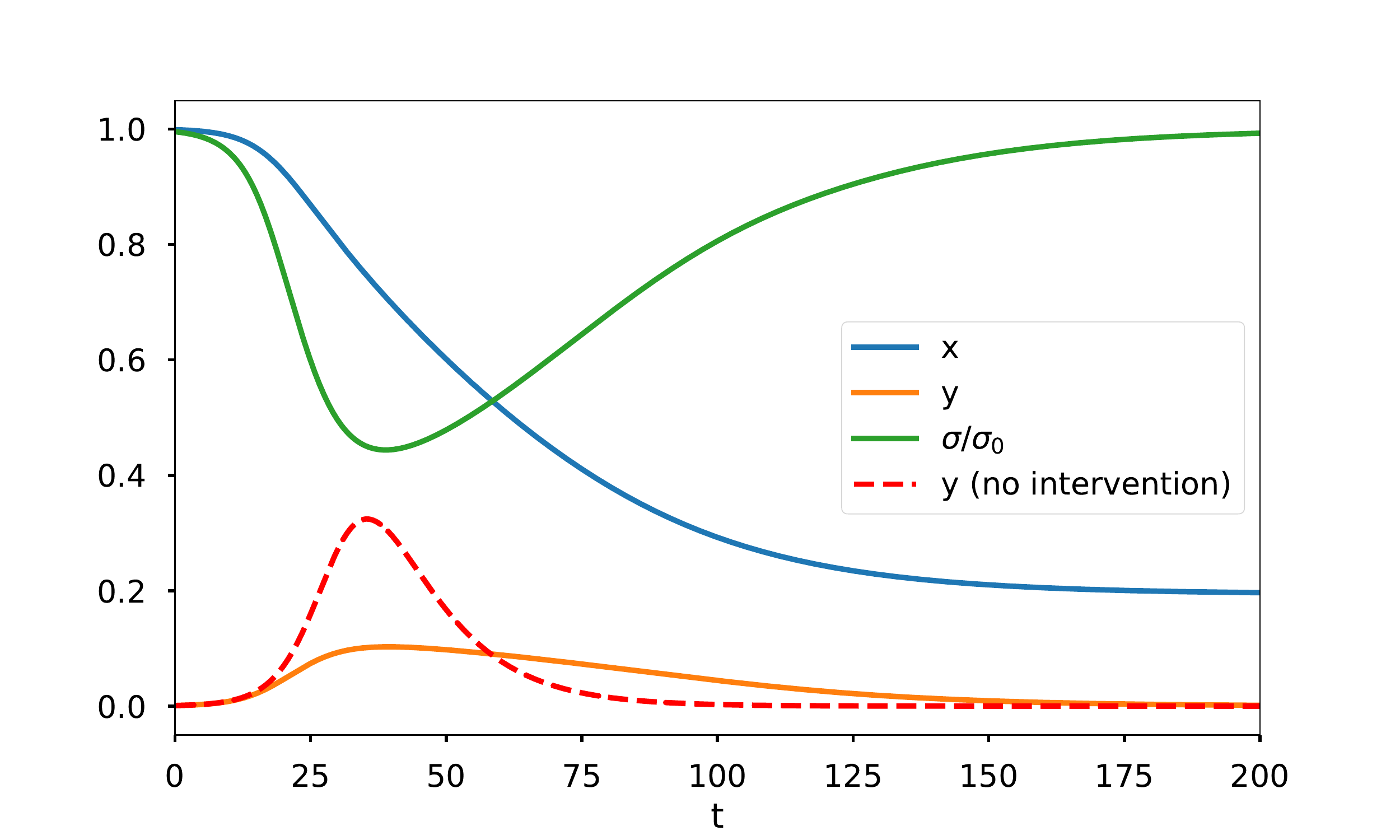}}
    \subfigure[Trajectory in phase space]{\label{fig:real_world_3-xy}\includegraphics[width=0.34\textwidth]{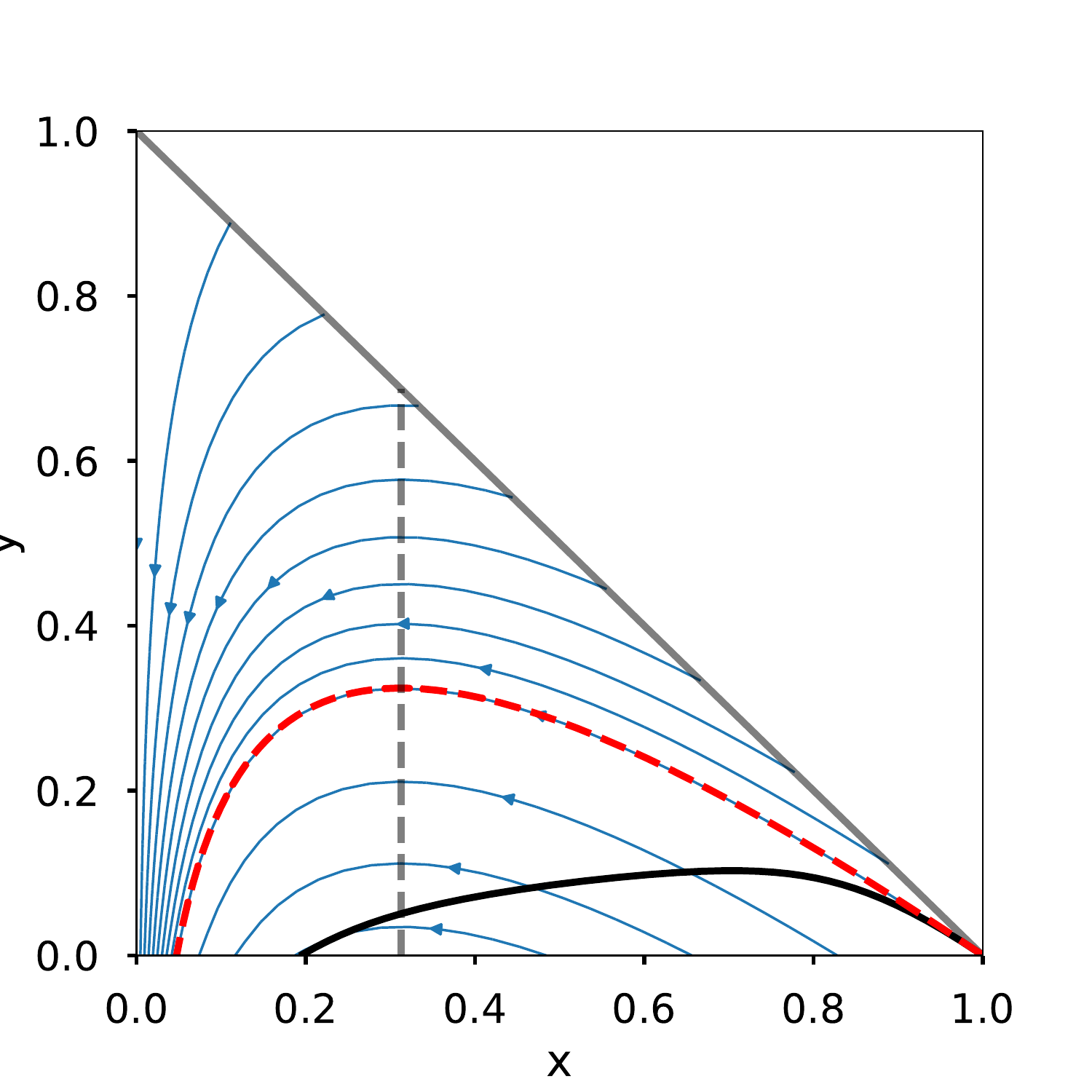}}
    \caption{Optimal control for COVID-19 with $\Rnot=3.2$, $\gamma=0.1$,
                $\alpha=\eta=0.006$, $\epsilon=0.5$, $d=10^4$, $T=200$, and $(x(0),y(0)) =
                (0.999,0.001)$.
                The dashed red line shows the result of imposing no control.
                \label{fig:real-world-3}}
\end{figure}

\section{Conclusion\label{sec:conclusion}}
We have studied, for an SIR model with a control on the rate of contact, the
problem of minimizing the eventually infected population in the long-time limit,
when the control can be applied only up to a finite time.  In the absence of any
cost of intervention, the optimal strategy is to apply no control until a
certain switching time, and then apply maximum control.  We have also considered
other objective functions that include a running cost of control and a penalty for
large numbers of simultaneous infections.

Contrary to simple intuition, it is not optimal to impose the maximum
level of intervention from the earliest possible time.  But real-world
studies have supported this observation; a too-strong intervention
may simply lead to a strong second wave of infection after
the intervention is lifted, and not significantly reduce epidemiological
overshoot \cite{bootsma2007effect}.  On the other hand, intervention that
starts too late or is lifted too soon may also have a negligible
effect on total mortality \cite{bootsma2007effect,hatchett2007public,markel2007nonpharmaceutical}.
The idea that intervention should possibly be delayed in order to increase
its effect was also found in \cite{ballard2017intervention}, although 
the objective and optimal policy found there differ from the present work.


The general results obtained here may provide insight
into what optimal intervention strategies and their consequences may look like,
but this should be informed by additional insight that can be gained
from more detailed models.
This work could form the basis of more detailed real-world application,
using values of the disease parameters, costs, and effectiveness of NPIs
relevant to a specific population of interest.  Although the SIR model
is perhaps the simplest mathematical epidemiological model available,
it has the advantage of requiring only a few parameters to be constrained.
Results based on control of the SIR model could form a starting point for
studying control in more complex models.

\bibliographystyle{plain}
\bibliography{refs}

\begin{thebibliography}{10}

\bibitem{agusto2013optimal}
F.~B. Agusto.
\newblock Optimal isolation control strategies and cost-effectiveness analysis
  of a two-strain avian influenza model.
\newblock {\em Biosystems}, 113(3):155--164, 2013.

\bibitem{ballard2017intervention}
PG~Ballard, NG~Bean, and JV~Ross.
\newblock Intervention to maximise the probability of epidemic fade-out.
\newblock {\em Mathematical Biosciences}, 293:1--10, 2017.

\bibitem{bootsma2007effect}
Martin~CJ Bootsma and Neil~M Ferguson.
\newblock The effect of public health measures on the 1918 influenza pandemic
  in {US} cities.
\newblock {\em Proceedings of the National Academy of Sciences},
  104(18):7588--7593, 2007.

\bibitem{clarke2013functional}
Francis Clarke.
\newblock {\em Functional analysis, calculus of variations and optimal
  control}, volume 264.
\newblock Springer Science \& Business Media, 2013.

\bibitem{ferguson2005strategies}
Neil~M. Ferguson, Derek A.~T. Cummings, Simon Cauchemez, Christophe Fraser,
  Steven Riley, Aronrag Meeyai, Sopon Iamsirithaworn, and Donald~S Burke.
\newblock Strategies for containing an emerging influenza pandemic in
  {S}outheast {A}sia.
\newblock {\em Nature}, 437(7056):209--214, 2005.

\bibitem{fister1998optimizing}
K~Renee Fister, Suzanne Lenhart, and Joseph~Scott McNally.
\newblock Optimizing chemotherapy in an {HIV} model.
\newblock {\em Electronic Journal of Differential Equations}, 1998(32):1--12,
  1998.

\bibitem{greenhalgh1988some}
David Greenhalgh.
\newblock Some results on optimal control applied to epidemics.
\newblock {\em Mathematical Biosciences}, 88(2):125--158, 1988.

\bibitem{harko2014exact}
Tiberiu Harko, Francisco S.~N. Lobo, and M.~K. Mak.
\newblock Exact analytical solutions of the susceptible-infected-recovered
  ({SIR}) epidemic model and of the {SIR} model with equal death and birth
  rates.
\newblock {\em Applied Mathematics and Computation}, 236:184--194, 2014.

\bibitem{hatchett2007public}
Richard~J. Hatchett, Carter~E. Mecher, and Marc Lipsitch.
\newblock Public health interventions and epidemic intensity during the 1918
  influenza pandemic.
\newblock {\em Proceedings of the National Academy of Sciences},
  104(18):7582--7587, 2007.

\bibitem{hethcote2000mathematics}
Herbert~W Hethcote.
\newblock The mathematics of infectious diseases.
\newblock {\em SIAM review}, 42(4):599--653, 2000.

\bibitem{jung2002optimal}
E~Jung, Suzanne Lenhart, and Z~Feng.
\newblock Optimal control of treatments in a two-strain tuberculosis model.
\newblock {\em Discrete \& Continuous Dynamical Systems-B}, 2(4):473, 2002.

\bibitem{kar2011stability}
T.~K. Kar and Ashim Batabyal.
\newblock Stability analysis and optimal control of an {SIR} epidemic model
  with vaccination.
\newblock {\em Biosystems}, 104(2-3):127--135, 2011.

\bibitem{kermack1927contribution}
William~Ogilvy Kermack and Anderson~G. McKendrick.
\newblock A contribution to the mathematical theory of epidemics.
\newblock {\em Proceedings of the Royal Society of London. Series A, Containing
  papers of a mathematical and physical character}, 115(772):700--721, 1927.

\bibitem{2013_sharpclaw}
D~I Ketcheson, Matteo Parsani, and R~J LeVeque.
\newblock {High-order Wave Propagation Algorithms for Hyperbolic Systems}.
\newblock {\em SIAM Journal on Scientific Computing}, 35(1):A351--A377, 2013.

\bibitem{ketcheson2021SIRRepro}
David~I Ketcheson.
\newblock {SIR-control-code}. {O}ptimal control of an sir epidemic through
  finite-time non-pharmaceutical intervention.
\newblock \url{https://github.com/ketch/SIR-control-code}, 5 2021.

\bibitem{2012_pyclaw-sisc}
David~I. Ketcheson, Kyle~T. Mandli, Aron~J. Ahmadia, Amal Alghamdi, Manuel
  {Quezada de Luna}, Matteo Parsani, Matthew~G. Knepley, and Matthew Emmett.
\newblock {PyClaw: Accessible, Extensible, Scalable Tools for Wave Propagation
  Problems}.
\newblock {\em SIAM Journal on Scientific Computing}, 34(4):C210--C231,
  November 2012.

\bibitem{kirschner1997optimal}
Denise Kirschner, Suzanne Lenhart, and Steve Serbin.
\newblock Optimal control of the chemotherapy of {HIV}.
\newblock {\em Journal of Mathematical Biology}, 35(7):775--792, 1997.

\bibitem{lenhart2007optimal}
Suzanne Lenhart and John~T Workman.
\newblock {\em Optimal control applied to biological models}.
\newblock CRC press, 2007.

\bibitem{liu2020reproductive}
Ying Liu, Albert~A Gayle, Annelies Wilder-Smith, and Joacim Rockl{\"o}v.
\newblock The reproductive number of {COVID}-19 is higher compared to {SARS}
  coronavirus.
\newblock {\em Journal of Travel Medicine}, 2020.

\bibitem{2016_clawpack}
Kyle~T Mandli, Aron~J Ahmadia, Marsha Berger, Donna Calhoun, David~L George,
  Yiannis Hadjimichael, David~I Ketcheson, Grady~I Lemoine, and Randall~J
  LeVeque.
\newblock Clawpack: building an open source ecosystem for solving hyperbolic
  {PDE}s.
\newblock {\em PeerJ Computer Science}, 2:e68, 2016.

\bibitem{markel2007nonpharmaceutical}
Howard Markel, Harvey~B Lipman, J~Alexander Navarro, Alexandra Sloan, Joseph~R
  Michalsen, Alexandra~Minna Stern, and Martin~S Cetron.
\newblock Nonpharmaceutical interventions implemented by {US} cities during the
  1918-1919 influenza pandemic.
\newblock {\em JAMA}, 298(6):644--654, 2007.

\bibitem{pakes2015lambert}
Anthony~G. Pakes.
\newblock Lambert's {W} meets {K}ermack--{M}c{K}endrick epidemics.
\newblock {\em IMA Journal of Applied Mathematics}, 80(5):1368--1386, 2015.

\bibitem{russell2020estimating}
Timothy~W Russell, Joel Hellewell, Christopher~I Jarvis, Kevin Van-Zandvoort,
  Sam Abbott, Ruwan Ratnayake, Stefan Flasche, Rosalind~M Eggo, Adam~J
  Kucharski, CMMID nCov~working group, et~al.
\newblock Estimating the infection and case fatality ratio for {COVID}-19 using
  age-adjusted data from the outbreak on the {D}iamond {P}rincess cruise ship.
\newblock {\em medRxiv}, 2020.

\bibitem{safi2013dynamics}
Mohammad~A. Safi and Abba~B. Gumel.
\newblock Dynamics of a model with quarantine-adjusted incidence and quarantine
  of susceptible individuals.
\newblock {\em Journal of Mathematical Analysis and Applications},
  399(2):565--575, 2013.

\bibitem{sharomi2017optimal}
Oluwaseun Sharomi and Tufail Malik.
\newblock Optimal control in epidemiology.
\newblock {\em Annals of Operations Research}, 251(1-2):55--71, 2017.

\bibitem{sun2020tracking}
Haoxuan Sun, Yumou Qiu, Han Yan, Yaxuan Huang, Yuru Zhu, and Song~Xi Chen.
\newblock Tracking and predicting {COVID}-19 epidemic in {C}hina {M}ainland.
\newblock {\em medRxiv}, 2020.

\bibitem{verity2020estimates}
Robert Verity, Lucy~C Okell, Ilaria Dorigatti, Peter Winskill, Charles
  Whittaker, Natsuko Imai, Gina Cuomo-Dannenburg, Hayley Thompson, Patrick~GT
  Walker, Han Fu, et~al.
\newblock Estimates of the severity of coronavirus disease 2019: a model-based
  analysis.
\newblock {\em The Lancet Infectious Diseases}, 2020.

\bibitem{wu2020estimating}
Joseph~T Wu, Kathy Leung, Mary Bushman, Nishant Kishore, Rene Niehus, Pablo~M
  de~Salazar, Benjamin~J Cowling, Marc Lipsitch, and Gabriel~M Leung.
\newblock Estimating clinical severity of {COVID}-19 from the transmission
  dynamics in {W}uhan, {C}hina.
\newblock {\em Nature Medicine}, pages 1--5, 2020.

\bibitem{yan2008optimal}
Xiefei Yan and Yun Zou.
\newblock Optimal and sub-optimal quarantine and isolation control in {SARS}
  epidemics.
\newblock {\em Mathematical and Computer Modelling}, 47(1-2):235--245, 2008.

\end{thebibliography}

\end{document}